\documentclass[11pt,reqno]{amsart}
\usepackage{color, amsmath,amssymb, amsfonts, amstext,amsthm, latexsym}
\usepackage[active]{srcltx}
\allowdisplaybreaks
\newcommand{\RR}{{\mathbb R}}
\newcommand{\HH}{{\mathbb H}}

\newcommand{\EE}{{\mathbb E}}

\newcommand{\LL}{{\mathbb L}}
  
\newcommand{\PP}{{\mathbb P}}

\numberwithin{equation}{section}
\newtheorem{theorem}{Theorem}[section]
\newtheorem{defn}[theorem]{Definition}

\newtheorem{lemma}[theorem]{Lemma}
\newtheorem{remark}[theorem]{Remark}
\newtheorem{prop}[theorem]{Proposition}

\begin{document}
\title[Strong Convergence of time Euler schemes for 3D NS]
{Strong $L^2$ convergence of time Euler schemes  
 \\ for stochastic  3D Brinkman-Forchheimer-Navier-Stokes equations}

\author[H. Bessaih]{Hakima Bessaih}
\address{Florida International University, Mathematics and Statistics Department, 11200 SW 8th Street, Miami, FL 33199, United States}
\email{hbessaih@fiu.edu}

\author[A. Millet]{ Annie Millet}
\address{SAMM, EA 4543,
Universit\'e Paris 1 Panth\'eon Sorbonne, 90 Rue de
Tolbiac, 75634 Paris Cedex France {\it and} Laboratoire de
Probabilit\'es, Statistique et Mod\'elisation, UMR 8001, 
  Universit\'es Paris~6-Paris~7} 
\email{amillet@univ-paris1.fr}


\subjclass[2000]{ Primary 60H15, 60H35; Secondary 76D06, 76M35.} 

\keywords{Stochastic Navier-Sokes equations, time Euler schemes, 
strong convergence, implicit time discretization,  Brinkman Forchheimer}

\begin{abstract}
We prove that some time Euler schemes for the 3D Navier-Stokes equations modified by adding a Brinkman-Forchheimer term and
 a random perturbation converge in $L^2(\Omega)$.
This extends previous results  concerning  the strong  rate
 of convergence of some time discretization schemes for the 2D Navier Stokes equations. 
Unlike  the 2D case,  our  proposed  3D model  with  the Brinkman-Forchheimer term  allows for  a strong 
 rate  of convergence 
of order almost 1/2,  that is independent of the viscosity parameter. 
\end{abstract}

\maketitle


\section{Introduction}\label{s1} \smallskip

An incompressible fluid flow dynamic  can be  described by the so-called incompressible Navier-Stokes equations  (NSEs). 
The fluid flow is defined by a velocity field  $u$ and a pressure term $\pi$ that evolve in a very particular way. 
These equations are parametrized by the viscosity coefficient $\nu>0$. 
Many questions are open in the 3D setting.
In this paper, we will focus on the 3D incompressible Navier-Stokes equations with  a smoothing term of Brinkman-Forchheimer type,  
 in a bounded domain   $D=[0,L]^3$ of $\RR^3$,  and 
subject to an external forcing defined as:
\begin{align} \label{3D-NS}
 \partial_t u - \nu \Delta u + (u\cdot \nabla) u + a |u|^{2\alpha} u + \nabla \pi & = G(u) dW\quad \mbox{\rm in } \quad (0,T)\times D,\\
 \mbox{\rm div }u&=0 \quad \mbox{\rm in } \quad (0,T)\times D,\nonumber 
 \end{align}
for $a>0$,  $\alpha \in [1,+\infty)$ and some terminal time  $T>0$.  The process  
$u: \Omega\times [0,T]\times D  \to \RR^3$  is  the velocity field 
 with initial condition $u_0$ in $D$,  and  periodic boundary conditions $u(t,x+L v_i)=u(t,x)$ on  $(0,T)\times \partial D$, 
 where $v_i$, $i=1,2, 3$ denotes the canonical basis of $\RR^3$,  and $\pi : \Omega\times [0,T]\times D  \to \RR$  is  the  pressure. 
 Note that similar computations using the restriction to a bounded domain as a technical step would enable to deal with $D=\RR^3$
 (with no boundary condition).
 In order to focus on the main issue, this will not be treated here.

  Here $G$ is a diffusion coefficient with global Lipschitz conditions and linear growth and the driving noise $W$ is a Wiener process defined
 of a filtered probability space  $(\Omega, {\mathcal F}, ({\mathcal F}_t),  \PP)$.
In 2D, there is an extensive  literature concerning the deterministic  NSEs
 and we refer to the books of Temam; see \cite{Tem, Temam-84} for known 
results. The stochastic setting has  also been widely investigated in dimension 2, see \cite{FG} for some very general results and the references therein. 
Unique global weak and strong solutions (in the PDE sense) are constructed for both additive and multiplicative noise, 
and without being exhaustive, we refer to \cite{Breckner,  ChuMil}. 

Global well posedness in the 3D case is a famous open problem, and can be proved with some additional smoothing term such
as either a Brinkman Forchheimer nonlinearity  to model porous media, or some  rotating fluid term. Let us mention that these  models
can be used with some anisotropic viscosity, that is no viscosity in one direction (see e.g. \cite{BTZ} and \cite{CDGG}). 
 The stochastic case has been investigated as well by several authors among which F.~Flandoli, M.~R\"ockner  and M.~Romito; 
 see for example \cite{F} for
 an account of remaining open problems. The anisotropic  3D case with a  stochastic perturbation 
has been studied in \cite{FM} for rotating fluids, and  in \cite{BeMi_anisotropic} for a Brinkman Forchheimer
modification.
\medskip

Numerical schemes and algorithms  were  introduced to best approximate and construct solutions for PDEs. 
A similar approach has started to emerge for stochastic models,   in particular SPDEs,  and has known a strong
 interest by the probability community.  
Many algorithms  based on either finite difference, finite elements  or spectral Galerkin methods 
(for the space discretization), and on either Euler, Crank-Nicolson or   splitting  schemes 
(for the temporal discretization) 
have been introduced for both the linear and nonlinear cases. Their rates of convergence have been widely investigated.
The literature on numerical analysis for SPDEs is now very extensive. 
 Models having either linear,  global Lipschitz properties or more generally some monotonicity properties are well developed in
an   extensive literature, see \cite{Ben1, Ben2}. In this case the convergence is proven to be in mean square. 
When nonlinearities are involved that are not of Lipschitz or monotone type,   a rate of  convergence 
 in mean square is more difficult to  obtain.  
  Indeed, because of the stochastic perturbation, there is no way of using the Gronwall lemma after taking the expectation of the
   error bound   
  because it involves a nonlinear term that is usually in a quadratic form. 
  One way of getting around it is to localize the nonlinear term in order to get a linear inequality, and then use the Gronwall lemma. 
  This gives rise to  a rate of   convergence in probability, that was first introduced by J.~Printems \cite{Pri}. 
\smallskip

Discretizations of the 2D stochastic Navier-Stokes equations with a multiplicative noise 
  were  investigated in several papers.  
 The following ones provide a rate of convergence in probability
 of time implicit Euler or splitting schemes \cite{BrCaPr}, \cite{CarPro}, \cite{Dor} and \cite{BeBrMi}. The Euler scheme is coupled with
 a finite element space discretization.  Note that \cite{Dor} tackles the problem of weak
 convergence, that is  convergence in distribution, while   in case of an additive noise \cite{Breckner} proves almost sure and mean square
  convergence without giving an explicit rate.
 
  Strong (i.e. $L^2(\Omega)$) convergence for a time splitting scheme, for an implicit time Euler scheme
 - coupled with a finite elements approximation -
of the stochastic 2D Navier-Stokes equations
  were   proven in \cite{BeMi_multi}, \cite{BeMi_FEM} for a multiplicative noise or ``additive" noise.
  In the latter case a polynomial  (suboptimal)  speed of
 convergence is proven. 

 In  \cite{BeMi_additive}, strong convergence  of  a space-time discretization (implicit Euler scheme  in time and finite elements approximation 
 in space) 
 for  stochastic 2D Navier-Stokes equations on the torus with an additive noise is studied.  
 The rate of convergence is  "optimal", namely almost 1/2 in time and 1 in space. However, since exponential moments of the $H^1$-norm
 of the solution is used, some constraints on the strength of the noise  have  to be imposed. 
 In the additive case, no localization is needed and the
 argument is based on a direct use of the discrete Gronwall lemma. 
  \medskip

 In this paper, we study a time implicit Euler scheme   \eqref{4.1}  for a stochastic 3D Navier Stokes equation 
 with a modification, by adding a smoothing term of Brinkman Forchheimer type. Unlike the 2D case - and thanks to this extra term - 
 neither localization nor exponential moments are   needed, and we obtain the ``optimal" 
convergence   rate  
 with no constraint on the noise and the viscosity. For technical reasons, we only have to  assume that the exponent $\alpha$ of the
Brinkman Forchheimer term $|u|^{2\alpha} u$ in \eqref{3D-NS} belongs to the interval $[1, \frac{3}{2}]$. 
 The proof is based on a careful study of the time regularity of the solution  in both  the $L^2$ and $H^1$ norms, and the discrete Gronwall lemma. 
 \smallskip

The paper is organized as follows. Section \ref{preliminary} describes the functional setting of the model. 
 In Section \ref{s-gwp} we describe  the stochastic perturbation,
state the global well posedness of the solution to \eqref{3D-NS} and its  moment estimates in various norms. If the exponent $\alpha =1$ we have
to impose that the coefficient $a$ is ``large". 
The way the Brinkman-Forchheimer term helps to obtain estimates for the bilinear part is described in Section \ref{Ap1} of the Appendix.
The proof of the existence and uniqueness relies on a Galerkin approximation. It is quite classical, 
 similar to  the anisotropic case described in \cite{BeMi_anisotropic}. The proof is sketched in Sections \ref{Ap2} and \ref{Ap3} of the Appendix 
 for the sake of completeness. 
Section \ref{s-increments} is devoted  to the moment time increments  of the solution to \eqref{3D-NS} in $L^2$ and $H^1$;
 the results are  crucial to obtain the
  optimal strong   convergence rate. In Section \ref{sEuler} we describe the fully
  implicit time Euler scheme, prove its existence and some moment
 estimates. Finally, in Section \ref{s-convergence} we prove the strong (that is $L^2(\Omega)$)  convergence  rate  of this scheme. 
\smallskip

As usual, except if specified otherwise, $C$ denotes a positive constant that may change  throughout the paper,
 and $C(a)$ denotes  a positive constant depending on some parameter $a$.

\section{Notations and preliminary results}\label{preliminary} 
Let  $D= [0,L]^3$  with periodic boundary conditions, 
${\mathbb L}^p:=L^p(D)^3$ (resp. ${\mathbb W}^{k,p}:=W^{k,p}(D)^3$)   be  the usual Lebesgue and Sobolev spaces of 
vector-valued functions
endowed with the norms  $\|\cdot \|_{\LL^p}$  (resp. $\|\cdot \|_{{\mathbb W}^{k,p}}$). If $p=2$, 
set ${\mathbb H}^k:={\mathbb W}^{k,2}$ and we denote
by $\| \, \cdot \|_k$ the ${\mathbb H}^k$ norm, $k=0,1, \cdots$; note that $\|\, . \, \|_0 = \| \, \cdot \, \|_{\LL^2}$. 
 In what follows, we will  consider velocity fields that have  zero divergence on $D$. 
  Let $H$ (resp. $V$)  be the subspace  of $\LL^2$ (resp. $\HH^1$) defined by 
\begin{align*}
  H:=& \{ u\in \LL^2 \, : \,{\rm div }\,  u=0 \; \mbox {\rm weakly in }  D  \; \mbox{\rm with periodic boundary conditions} \}, \\ 
  V:= & H \cap {\mathbb W}^{1,2}.
  \end{align*}
   $H$ and $V$ are   separable  Hilbert spaces. 
      The space $H$ inherits its inner product  denoted by $(\cdot,\cdot)$ and its norm $\|\, \cdot\, \|_H$ from $\LL^2$.
   The norm in $V$, inherited from ${\mathbb W}^{1,2}$, is denoted by $\| \cdot \|_V$; we let $(\cdot ,\cdot)_V$ denote the associated inner product.   
   Moreover,   let  $V'$ be the dual space of $V$ 
   with respect to the  pivot space $H$,  and 
  $\langle\cdot,\cdot\rangle$ denotes the duality between $V'$ and $V$.   
  
  Let  $\Pi : \LL^2 \to H$ denote the Leray projection, and set  $A=- \Pi \Delta$ with its domain  
  $\mbox{\rm Dom}(A)={\mathbb W}^{2,2}\cap H$.

  Let $b:V^3 \to \RR$ denote the trilinear map defined by 
  \[ b(u_1,u_2,u_3):=\int_D  \big(\big[ u_1(x)\cdot \nabla\big] u_2(x)\big)\cdot u_3(x)\, dx, \]
  which by the incompressibility condition satisfies  $b(u_1,u_2, u_3)=-b(u_1,u_3,u_2)$  for $u_i \in V$, $i=1,2,3$. 
  There exists a continuous bilinear map $B:V\times V \mapsto
  V'$ such that
  \[ \langle B(u_1,u_2), u_3\rangle = b(u_1,u_2,u_3), \quad \mbox{\rm for all } \; u_i\in V, \; i=1,2,3.\]
  The map  $B$ satisfies the following antisymmetry relations:
  \begin{equation} \label{B}
  \langle B(u_1,u_2), u_3\rangle = - \langle B(u_1,u_3), u_2\rangle , \quad \langle B(u_1,u_2),  u_2\rangle = 0 
  \qquad \mbox {\rm for all } \quad u_i\in V.
  \end{equation}
  For $u,v\in V$,  set $B(u,v):= \Pi \big( \big[ u\cdot \nabla\big]  v\big) $.

  In dimension 3, the Gagliardo-Nirenberg inequality implies that for $p\in [2,6]$, ${\mathbb H}^1\subset \LL^p$; more precisely 
    \begin{equation} \label{GagNir}
  \|u\|_{\LL^4} \leq \bar{C}_4 \; \|u\|_{\LL^2}^{\frac{1}{4}}  \, \|\nabla u\|_{\LL^2}^{\frac{3}{4}}  \quad \mbox{\rm and} \quad 
   \|u\|_{\LL^3} \leq \bar{C}_3 \; \|u\|_{\LL^2}^{\frac{1}{2}}  \, \|\nabla u\|_{\LL^2}^{\frac{1}{2}} , \quad \forall u\in {\mathbb H}^1, 
  \end{equation}
  for some positive constants $\bar{C}_3$ and $\bar{C}_4$. 
  
  Furthermore, the Gagliardo-Nirenberg inequality  implies that ${\mathbb H}^2\subset \LL^p$
  for any $p\in [2,\infty)$, and for $u\in \HH^2$
  \begin{equation} \label{GagNirH2}
  \|u\|_{\LL^p} \leq C(p) \;  \|Au\|_{\LL^2}^{\beta(p)} \, \|u\|_{\LL^2}^{1-\beta(p)}\quad \mbox{\rm for}\quad \beta(p) 
  = \frac{3}{2} \Big( \frac{1}{2}-\frac{1}{p}\Big).
  \end{equation}
  Note that for $p=6$ we have $\beta(6)=\frac{1}{2}$. Furthermore, $\|u\|_{\LL^\infty} \leq C \|u\|_{\HH^2}$ for
  $u\in \HH^2$. 
  
  Let $\alpha \in (1,+\infty)$ and  let $f,g, h:  D \to \RR$ be regular functions.
  Given  any positive constants $\varepsilon_0$ and $\varepsilon_1$
  and some constant $C_\alpha$ 
 depending on $\alpha$, the following upper estimates are straightforward consequences of the H\"older and Young inequalities
    \begin{align}
  \int_D \big| f(x) g(x) h(x) \big| dx & \leq \big\| |f| |g|^{\frac{1}{\alpha}} \big\|_{L^{2\alpha}} \, 
  \big\| |g|^{1-\frac{1}{\alpha}} \big\|_{L^{\frac{2\alpha}{\alpha -1}}}\, \|h\|_{L^2}  \label{fgh1}. \\
  & \leq \epsilon_0 \|h\|_{L^2}^2 + \frac{\varepsilon_1}{4\varepsilon_0} \big\| |f|^\alpha g\big\|_{L^2}^2 + \frac{C_\alpha}{\varepsilon_0
  \varepsilon_1^{\frac{1}{\alpha -1}}} \|g\|_{L^2}^2.  \label{fgh2}
  \end{align}

Let $\Omega_T=\Omega\times [0,T]$ be endowed with the 
product measure $d\PP \otimes ds$ on ${\mathcal F}\otimes {\mathcal B}(0,T)$. 
The following functional notations will be used throughout the paper. Set 
\begin{align}
X_0=&\; L^\infty(0,T;H) \cap L^2(0,T;V) \cap L^{2\alpha +2}([0,T]\times D; \RR^3), \label{X0}\\
{\mathcal X}_0=&\; L^4\big(\Omega ; L^\infty(0,T;H)\big) \cap L^2\big( \Omega ; L^2(0,T;V)\big) \cap L^{2\alpha +2}(\Omega_T \times D;\RR^3), 
\label{calX}\\
X_1=& \;L^\infty(0,T;V) \cap L^2(0,T;\mbox{\rm Dom} A) \cap \Big\{ u: [0,T]\times D \to \RR^3 : \nonumber \\
&\qquad \qquad \qquad \int_0^T \big[ \|u(t)\|_{\LL^{2\alpha +2}}^{2\alpha +2}
+ \big\| |u(t)|^\alpha \nabla u(t) \big\|_{\LL^2}^2 \big] dt <\infty \Big\} , \label{X1}\\
{\mathcal X}_1=&\;  L^4\big(\Omega ; L^\infty(0,T;V)\big) \cap L^2\big( \Omega ; L^2(0,T; \mbox{ \rm Dom } A)\big) \cap
 \Big\{ u: \Omega_T \times D \to \RR^3 : \nonumber \\
 & \qquad \qquad \qquad \EE \int_0^T \big[ \|u(t)\|_{\LL^{2\alpha +2}}^{2\alpha +2}
+ \big\| |u(t)|^\alpha \nabla u(t) \big\|_{\LL^2}^2 \big] dt  <\infty \Big\}. \label{calX1}
\end{align}

\section{Global well posedness and first moment estimates} \label{s-gwp}
For technical reasons, we assume that the initial condition $u_0$ belongs to $L^p(\Omega ; V)$ for some $p\in [2,\infty]$, 
   and  only consider {\it strong solutions}  in the PDE sense. We prove that the stochastic 3D Navier-Stokes equation with 
   Brinkman-Forchheimer smoothing \eqref{3D-NS} 
   has a unique solution on any time interval $[0,T]$ and prove moment estimates
of this solution. This requires some hypotheses on the driving noise $W$ and the diffusion coefficient $G$.

\subsection{The driving noise and the diffusion coefficient}	\label{def-WG}
 Let  $(e_k, k\geq 1)$ be an orthonormal basis of $H$ whose elements belong to $\HH^2:=W^{2,2}(D; \RR^3)$ and
are orthogonal in $V$.  Let ${\mathcal H}_n= \mbox{\rm span }(e_1, \cdots, e_n)$ and let
$P_n$ (resp. $\tilde{P}_n$) denote the orthogonal projection from $H$ (resp. $V$)   onto  ${\mathcal H}_n$. 
Furthermore, given $i\neq j$ we have
\[ (A e_i\, , \, e_j) = (\nabla e_i\, , \, \nabla e_j) =0\]
since the basis $\{ e_n\}_n$ is orthogonal in $V$. Hence $Au\in {\mathcal H}_n$ for every $u\in {\mathcal H}_n$. 

We deduce that for $u\in V$ we have $P_n u = \tilde{P}_n u$. Indeed,  for  $v\in {\mathcal H}_n$
and  $u\in V$: 
\begin{equation}		\label{grad_Pn}
 (P_n u, v) = (u,v), \quad \mbox{\rm and }\;  (\nabla P_n u, \nabla v) = - (P_n u, A v) =-(u, Av) = (\nabla u, \nabla v).
 \end{equation}
Hence  given $u\in V$, we have $(P_n u, v)_V=(u,v)_V$ for any $v\in {\mathcal H}_n$. 
  
  Let $K$ be a separable Hilbert space and $Q$ be a symmetric, positive trace-classe operator on $K$. 
  Let $(W(t), t\in [0,T])$ be a $K$-valued Wiener process with covariance operator $Q$, defined on the probability
  space $(\Omega, {\mathcal F},  ({\mathcal F}_t),  \PP)$. Let $\{\zeta_j\}_{j\geq 1}$ denote an orthonormal basis of $K$
   made of eigenfunctions of $Q$, with eigenvalues $\{q_j\}_{j\geq 1} $ and ${\rm Tr} Q = \sum_{j\geq 1} q_j <\infty$. Then 
   \[ W(t)=\sum_{j=1}^\infty \sqrt{q_j} \, \beta^j(t)\, \zeta_j, \qquad \forall t\in [0,T],\]
   where $\{ \beta_j\}_{j\geq 1}$ are independent one-dimensional Brownian motions defined on $(\Omega, {\mathcal F},  ({\mathcal F}_t),  \PP)$.

For details concerning this Wiener process  we refer  to \cite{DaPZab}.

 
Let ${\mathcal L} \equiv {\mathcal L}(K ; H)$ (resp. $\widetilde{\mathcal L} \equiv {\mathcal L}(K ; V)$) 
 be the space of continuous  linear operators  from $K$ to $H$ (resp. $V$) with norm $\|\,.\,\|_{\mathcal L}$ (resp.
 $\|\,.\,\|_{\widetilde{\mathcal L} }$). 

The noise intensity of the stochastic perturbation
 $G: V \to \widetilde{\mathcal L}$ which we put in \eqref{3D-NS} satisfies the following classical 
 growth and Lipschitz conditions (i) and (ii). Note that due to the 3D framework, we have to impose growth conditions
 both  on  the $\|\, \cdot \, \|_{\mathcal L}$ and  $\|\, \cdot \, \|_{\tilde{\mathcal L}}$ norms. 
 
\par
  The diffusion coefficient $G$ satisfies the following assumption:\\
 { \bf Condition (G)} Assume that 
   $G:V\to \tilde{{\mathcal L}}$  satisfies the following conditions:
  
  (i) {\bf Growth condition} There exist positive constants  $K_i$,  $\tilde{K}_i$,  $i=0,1$, 
  such that 
  \begin{align} 
\| G(u)\|_{\mathcal L}^2  & \leq K_0 + K_1 \|u\|_{H}^2, \quad \forall u\in H, \label{growthG_H}  \\ 
\| G(u)\|_{\tilde{\mathcal L}}^2 & \leq  \tilde{K}_0 + \tilde{K}_1 \|u\|_{V}^2, \quad \forall u\in V. \label{growthG_V}
  \end{align} 

 (ii) {\bf Lipschitz condition} There exists a positive constant $L$ such that 
 \begin{equation} 			\label{LipG}
  \|G(u)-G(v)\|_{\mathcal L}^2 \leq  L  \|u-v||_H^2 , \quad \forall u,v\in H. 
 \end{equation}

  We define a   weak pathwise solution (that is strong probabilistic solution in the weak deterministic sense)
  of \eqref{3D-NS} as follows: 
 \begin{defn}
 We say that equation \eqref{3D-NS} has a strong  solution if:
 \begin{itemize}
 \item  $u $ is an adapted $V$-valued process which belongs a.s. to $X_1$,
 \item $\PP$ a.s. we have  $u\in C([0,T];V)$, and 
  \begin{align*}
  \big(u(t), \phi\big) +& \nu \int_0^t \big( \nabla u(s), \nabla \phi\big) ds + \int_0^t \big\langle [u(s) \cdot \nabla]u(s), \phi\big\rangle ds \\
 &  + a \int_0^t  \int_ D |u(s,x)|^{2\alpha} u(s,x) \phi(x) dx ds  =
 \big( u_0, \phi) + \int_0^t \big( \phi ,  G(u(s)) dW(s) \big)
 \end{align*}
for every $t\in [0,T]$ and every $\phi \in V$.
  \end{itemize}
 \end{defn}
 
 \subsection{Global well-posedness and moment estimates of the solution}
 We next prove that if $\EE(\|u_0\|_V^4)<\infty$, then \eqref{3D-NS} has a unique solution $u$ in ${\mathcal X}_1$.
\begin{theorem}		\label{th_gwp}
 Let $\alpha \in [1,+\infty)$, and for $\alpha=1$ suppose that $4\nu a >1$. 
Let $u_0\in L^{2p}(\Omega;V)$,   for some $p\in [1,\infty)$, be independent of $W$, and $G$ satisfy the growth and Lipschitz conditions {\bf (G)}.
 Then equation \eqref{3D-NS}
has a unique solution in ${\mathcal X}_1$  such that a.s. $u\in C([0,T];V)$. 
Furthermore,
\begin{align}		\label{2.18}
\EE\Big( \sup_{t\in [0,T]} \|u(t)\|_V^{2p} +  \int_0^T \!\! \|A u(t)\|_{\LL^2}^2  dt   +& \int_0^T \!\!  \|u(t)\|_{\LL^{2\alpha +2}}^{2\alpha +2} 
dt \Big)
\leq C \big[1+\EE\big(\|u_0\|_V^{2p}\big) \big].
\end{align}
\end{theorem} 
The proof, which is quite classical, is sketched in Section \ref{Ap3} of the Appendix.

\section{ Moment estimates of time increments of the solution} \label{s-increments}
In this section we prove moment estimates for various norms of time increments of the solution to \eqref{3D-NS}. 
This will be crucial to deduce the speed of convergence of numerical schemes. Let 
$u_0\in L^{2p}(\Omega;V)$ for some $p\in [2,\infty)$ and $u$ be the solution to \eqref{3D-NS}, that is 
\begin{align} 			\label{3.1}
u(t) = & \, S(t) u_0 - \int_0^t \!\! S(t-s) B(u(s),u(s)) ds -a \int_0^t \!\! S(t-s) \Pi |u(s)|^{2\alpha} u(s) ds \nonumber \\
&\quad  + \int_0^t \!\!S(t-s) G(u(s)) dW(s), \quad \forall t\in [0,T], \quad \PP\;  \mbox{\rm a.s.}
\end{align} 
where $S(t)=e^{-\nu t A}$ is the analytic semi group  generated by the Stokes operator $A$ multiplied by the viscosity $\nu$. 
Then (see e.g. \cite{CarPro}, Lemma 2.2   and \cite{Pri}, Lemma 2.1), for $b>0$ and $t\in [0,T]$,
\begin{align}
\big\| A^b e^{-\nu tA}\big\|_{{\mathcal L}(\LL^2;\LL^2)} &\leq  C(b,\nu)\,  t^{-b}, 	\label{3.2}\\
\big\|A^{ -b} \big( \mbox{\rm Id } - e^{-\nu t A} \big) \big\|_{{\mathcal L}(\LL^2;\LL^2)}& \leq \widetilde{C(b,\nu)}\,  t^b, 	\label{3.3}
\end{align}
for some positive constants  $C(b,\nu)$ and $\widetilde{C(b,\nu)}$. 

The following regularity result for the bilinear term 
 will be crucial in the proof of time regularity.
\begin{lemma}		\label{Lem3.1}
(i)  There exists a positive constant $M$ such that 
\begin{equation}		\label{3.5}
\| A^{-\frac{1}{4}} B(u,u)\|_{\LL^2} \leq M \| A^{\frac{1}{2}} u\|_{\LL^2}^2 \leq  M \|u\|_V^2,\quad \forall  u\in V.
\end{equation}

(ii) For  $\delta \in (0, \frac{3}{4})$, 
\begin{equation} 		\label{3.5Bis}
\| A^{-\delta} B(u,u)\|_{\LL^2} \leq C  
\|A u\|_{\LL^2}^{\frac{3}{4}-\delta} \, \| u\|_{\HH^1}^{\frac{5}{4}+\delta},\quad \forall u\in  \mbox{\rm Dom}(A). 
\end{equation}
\end{lemma}
\begin{proof}
(i) Using \cite[Lemma 2.2]{GigMiy} we deduce that given  positive constants $\delta, \theta, \rho$ such that 
$0\leq \delta < \frac{1}{2}+\frac{3}{4}$, $\theta >0$, $\rho >0$ such that
$\rho + \delta > \frac{1}{2}$ and $\delta + \theta + \rho \geq \frac{5}{4}$, there exists a constant $M:=M(\delta,\theta,\rho)$ such that
for $u,v$ regular enough
\[ 
\| A^{-\delta} B(u,v)\|_{\LL^2} \leq M \| A^\theta u\|_{\LL^2} \|A^\rho v \|_{\LL^2}.
\] 
Choosing $\delta = \frac{1}{4}$, $\theta = \rho = \frac{1}{2}$, we deduce \eqref{3.5}.

(ii) For $u\in \HH^2$, we have
\[ \|A^{-\delta} B(u,u)\|_{\LL^2} = \sup\Big\{ \int_{D} |\nabla u|\, |u|\, |\phi| \, dx ; : \; \|\phi\|_{\HH^{2\delta}}\leq 1\Big\}.\]
 In dimension 3, the Sobolev embedding theorem (see e.g. \cite{Adams}, Theorem 7.57 page  217) implies
 $\mathbb{W}^{\beta,p}(D) \subset \LL^q(D)$ if
$3>\beta p$, $\beta >0$, $1<p<3$ and $p\leq q\leq \frac{3p}{3-\beta p}$. Hence for $\delta \in (0, \frac{3}{4})$, choosing $\beta =2\delta $, 
$p=2$ and $q=\frac{6}{3-4\delta}$, we obtain
${\mathbb W}^{2\delta ,2}(D) = \HH^{2\delta}(D)\subset \LL^q(D)$. Let $\bar{p}=\frac{3}{2\delta}$; then 
 $\frac{1}{\bar{p}} + \frac{1}{2}+\frac{1}{q} =1$, 
  and the H\"older inequality yields 
  \[ \|A^{-\delta} B(u,u)\|_{\LL^2} \leq C \| \nabla u\|_{\LL^2} \|u\|_{\LL^{\bar{p}}}.
  \]
Since the Gagliardo Nirenberg inequality   \eqref{GagNirH2}  implies 
$\|u\|_{\LL^{\bar{p}}} \leq C \| Au\|_{\LL^2}^{\frac{3}{4}-\delta} \|u\|_{\LL^2}^{\frac{1}{4}+\delta}$, 
 this concludes the proof of \eqref{3.5Bis}. 
\end{proof}

The following result proves regularity of the Brinkman-Forchheimer term. To have a regularity similar to that of the bilinear term, we have
to impose some restriction on the exponent $\alpha$.
\begin{lemma}			\label{Lem3.2}
Let   $\alpha \in [1,\frac{3}{2}]$. \\
\indent (i)  there exists a positive constant $C$ such that 
\begin{equation} 			\label{3.8}
\big\| A^{-\frac{1}{4}} \big( |u|^{2\alpha} u\big) \big\|_{\LL^2} \leq C \|u\|_V^{2\alpha +1}, \quad \forall u\in V. 
\end{equation}
\indent  (ii) Furthermore, for any  $\delta \in (0, \frac{3}{4})$ there exists $C>0$ such that
\begin{equation}		\label{3.8Bis}
\big\| A^{-\delta} \big( |u|^{2\alpha} u\big) \big\|_{\LL^2} \leq C  \| Au\|_{\LL^2}^{\frac{3}{4}-\delta} 
\|u\|_V^{2\alpha+\frac{1}{4}+\delta} 
 \quad \forall u\in {\rm Dom}(A).
\end{equation} 
\end{lemma}
\begin{proof}
We use once more the Sobolev embedding theorem  $\mathbb{W}^{\beta,p}(D) \subset \LL^r(D)$ if
$3>\beta p$, $\beta >0$, $1<p<3$ and $p\leq r\leq \frac{3p}{3-\beta p}$. 

(i) Choosing $\beta =\frac{1}{2}$, $p=2$ and $r=3$, we obtain
${\mathbb W}^{\frac{1}{2},2}(D) = H^{\frac{1}{2}}(D)\subset \LL^3(D)$, 
while $\beta=1$, $p=2$ and $r\in [2,6]$ yields $H^1(D)\subset \LL^r(D)$.
Given $u\in \HH^1$, we have  
\[ \big\|A^{-\frac{1}{4}} \big( |u|^{2\alpha} u\big) \big\|_{\LL^2} 
=  \sup\Big\{ \int_{D}  |u(x)|^{2\alpha} u(x)  \phi(x) dx  :  \| \phi\|_{\HH^{\frac{1}{2}}} \leq 1\Big\}.
\]
Using H\"older's inequality with exponents 2,6  and 3, we obtain for $\delta \in [\frac{1}{4}, \frac{3}{4})$
\[ \big\|A^{-\frac{1}{4}} \big( |u|^{2\alpha} u\big) \big\|_{\LL^2}  \leq \sup\{ \|u\|_{\LL^{4\alpha}}^{2\alpha} \|u\|_{\LL^6} \|\phi\|_{\LL^3} \, : \, 
\| \phi\|_{\HH^{\frac{1}{2}}} \leq 1\} \leq C \|u\|_V^{2\alpha +1},
\]
where the last upper estimate is a consequence of the inequality $4\alpha \in [4,6]$. This completes the proof of \eqref{3.8}.

(ii) As in the proof of Lemma \ref{3.1}  (ii)  we choose $q=\frac{6}{3-4\delta}$ to ensure $\HH^{2\delta}\subset \LL^q$ and $p=\frac{3}{2\delta}$.
The H\"older and Gagliardo Nirenberg inequalities imply 
\begin{align*}
 \|A^{-\delta} (|u|^{2\alpha} u)\|_{\LL^2} = &\sup\Big\{ \int |u|^{2\alpha} |u| \phi : \|\phi\|_{\HH^{2\delta}} \leq 1\Big\} \leq  
\|u\|_{\LL^{4\alpha}}^{2\alpha} \|u\|_{\LL^p}\|\phi\|_{\LL^q} \\
\leq &C \|u\|_{\LL^{4\alpha}}^{2\alpha} \|Au\|_{\LL^2}^{\frac{3}{4}-\delta} \|u\|_{\LL^2}^{\frac{1}{4}+\delta}. 
\end{align*}
Since $\alpha \in [1,\frac{3}{2}]$, the Sobolev embedding $\HH^1\subset  \LL^\gamma$ for $ \gamma\in [4,6]$ 
concludes the proof. 
\end{proof}

The following proposition gives upper estimates for moments of time increments of the solution to the stochastic 3D modified 
Navier Stokes equation $u$ defined in  equation \eqref{3.1}.

\begin{prop}			\label{time-increments}
Let $u_0$ be ${\mathcal F}_0$-measurable and  let  $\alpha \in [1,\frac{3}{2}]$ with $4\nu a >1$ if $\alpha =1$.   
Suppose that the diffusion coefficient $G$ satisfies Condition  {\bf   (G)} and let $u$ be the solution to \eqref{3D-NS}. 
Then for $\lambda \in (0,\frac{1}{2})$ we have

(i) Suppose  $u_0\in L^{(2\alpha+1) p}(\Omega ; V)$ for  some $p\in[2,\infty)$. 
There exists a positive constant  $C:= C(T,a,p, \mbox{\rm Tr}Q)$ such that for $0\leq t_1<t_2\leq T$,
\begin{equation}		\label{3.9}
\EE\big( \| u(t_2)-u(t_1) \|_H^{p}\big) \leq C 
 \, |t_2-t_1|^{\lambda p} \big[ 1+\EE\big(\|u_0\|_V^{(2\alpha +1)p}\big) \big].
\end{equation}

(ii) Let $N\geq 1$ be an integer  and for $k=0, \cdots, N$ set $t_k=\frac{kT}{N}$. Then there exists 
$C:=C(T,  a,  \mbox{\rm Tr}Q, \lambda)>0$ (independent of $N$) such that for  $p(\lambda)=\frac{2+8\alpha-2\lambda}{1-\lambda}$ and 
$u_0\in \LL^{p(\lambda)}(\Omega ; V) $ 
\begin{align}			\label{3.10}
\EE\Big(   \sum_{j=1}^N \int_{t_{j-1}}^{t_j}\!\! & \big[ \| \nabla (u(s) - u(t_j))\|_{\LL^2}^{2} +  \| \nabla (u(s) - u(t_{j-1}))\|_{\LL^2}^{2} \big] ds \Big)
\nonumber \\
 &\leq C \Big( \frac{T}{N}\Big)^{2\lambda}  \, 
\Big[ 1+\EE\Big(\|u_0\|_V^{p(\lambda)}\Big) \Big].
\end{align} 
\end{prop}
\begin{proof}
The proof relies on a semi-group argument.

(i) Let $t_1<t_2$ belong to the time interval $[0,T]$. Then $u(t_2)-u(t_1) = \sum_{i=1}^4 T_i$, where
\begin{align*}
T_1=& S(t_2) u_0 - S(t_1) u_0, \\
T_2=&-\int_0^{t_2} S(t_2-s) B(u(s),u(s)) ds + \int_0^{t_1} S(t_1-s) B(u(s),u(s)) ds, \\
T_3=&-a \int_0^{t_2} S(t_2-s) \big( |u(s)|^{2\alpha} u(s) \big) ds + a \int_0^{t_1} S(t_1-s) \big( |u(s)|^{2\alpha} u(s) \big) ds, \\
T_4=&\int_0^{t_2} S(t_2-s) G(u(s)) dW(s) - \int_0^{t_1} S(t_1-s) G(u(s)) dW(s).
\end{align*}
Then using \eqref{3.3} and the upper estimate $\sup_{t\in [0,T]} \|S(t)\|_{{\mathcal L}(\LL^2;\LL^2)} <\infty$ we deduce
\begin{align*}
 \|T_1\|_{\LL^2} =& \big\| S(t_1) A^{-\frac{1}{2}} \big[ S(t_2-t_1) - \mbox{\rm Id } \big] A^{\frac{1}{2}} u_0 \big\|_{\LL^2} \\
 & \leq C \|S(t_1)\|_{{\mathcal L}(\LL^2;\LL^2)} \, |t_2-t_1|^{\frac{1}{2}}  \|A^{\frac{1}{2}} u_0\|_{\LL^2} \leq C \, |t_2-t_1|^{\frac{1}{2}} \|u_0\|_V.
 \end{align*}
Hence taking expected values, we deduce for every $p\in [2,\infty)$
\begin{equation} 				\label{ET1}
\EE\big( \|T_1\|_{\LL^2}^p\big) \leq C^p \, |t_2-t_1|^{\frac{p}{2}} \, \EE(\|u_0\|_V^p). 
\end{equation}
Furthermore, $T_2=-T_{2,1}-T_{2,2}$, where
\begin{align*}
T_{2,1}=&\! \int_0^{t_1} \!\! S(t_1-s) \big[ S(t_2-t_1) - \mbox{\rm Id } \big] B(u(s),u(s)) ds, \;  T_{2,2} 
=\!\int_{t_1}^{t_2} \!\! S(t_2-s)  B(u(s),u(s)) ds.
\end{align*}
Using the Minkowski inequality, \eqref{3.2}, \eqref{3.3} and \eqref{3.5}, we deduce that for $\varepsilon \in \big(0, \frac{1}{4} \big) $, 
\begin{align*}
\| T_{2,1}\|_{\LL^2} \leq &\int_0^{t_1} \| A^{1-\epsilon} S(t_1-s) \, A^{-(\frac{3}{4}-\varepsilon)} \big[ S(t_2-t_1) - \mbox{\rm Id } \big]\, 
A^{-\frac{1}{4}} B(u(s),u(s)) \|_{\LL^2} ds \\
\leq &\;  C \, |t_2-t_1|^{\frac{3}{4} - \varepsilon} \sup_{s\in [0,t_1]} \|u(s)\|_V^2 \int_0^{t_1} (t_1-s)^{-1+\varepsilon} ds. 
\end{align*}
Hence \eqref{2.18} implies that if $\EE(\|u_0\|_V^{2p})<\infty$ for some $p\in [1,\infty)$, we have
\begin{equation}			\label{3.15}
\EE \big( \|T_{2,1}\|_{\LL^2}^{p} \big) \leq C(T)\, |t_2-t_1|^{(\frac{3}{4} - \varepsilon)p} \big[ 1+  \EE \|u_0\|_V^{2p}\big]. 
\end{equation}
The Minkowski inequality, \eqref{3.2} and \eqref{3.5} imply
\[ \|T_{2,2}\|_{\LL^2} \leq \int_{t_1}^{t_2} \!\!\| A^{\frac{1}{4}} S(t_2-s) A^{-\frac{1}{4}} B(u(s),u(s)) \|_{\LL^2} ds \leq C \!
\sup_{s\in [t_1,t_2]} \|u(s)\|_V^2 \int_{t_1}^{t_2} \!\! (t_2-s)^{-\frac{1}{4}} ds .\]
Using once more \eqref{2.18} we deduce that if $\EE(\|u_0\|_V^{2p})<\infty$ for some  $p\in [1,\infty)$,
\begin{equation}			\label{3.16}
\EE\big( \|T_{2,2}\|_{\LL^2}^{p}\big) \leq   C\, |t_2-t_1|^{\frac{3}{4} p} \big[ 1+  \EE \|u_0\|_V^{2p}\big]. 
\end{equation}
A similar decomposition yields $T_3=-a\big( T_{3,1}+T_{3,2}\big)$, where
\[ T_{3,1}=\int_0^{t_1} \!\! S(t_1-s) \big[ S(t_2-t_1) - \mbox{\rm Id } \big] |u(s)|^{2 \alpha} u(s) ds, \quad T_{3,2}=\int_{t_1}^{t_2} \!\! 
S(t_2-s) |u(s)|^{2\alpha} u(s) ds.
\]
The Minkowski inequality and the upper estimates \eqref{3.2}, \eqref{3.3} and  \eqref{3.8} imply that for $\varepsilon \in \big( 0, \frac{1}{4} \big)$, 
\begin{align*}
\|T_{3,1}\|_{\LL^2} \leq & \int_0^{t_1} \| A^{1-\varepsilon} S(t_1-s) \, A^{-(\frac{3}{4}-\varepsilon)} \big[ S(t_2-t_1) - \mbox{\rm Id } \big]
\, A^{-\frac{1}{4}} \big( |u(s)|^{2\alpha} u(s)\big)  \|_{\LL^2} ds\\
\leq & \; C |t_2-t_1|^{\frac{3}{4} - \varepsilon} \, \sup_{s\in [0,t_1]} \|u(s)\|_V^{2\alpha +1} \int_0^{t_1} (t_1-s)^{-(1-\varepsilon)} ds, 
\end{align*}
and the upper estimate \eqref{2.18} implies that for $p \in [1,\infty)$, 
\begin{equation}			\label{3.17}
\EE\big( \|T_{3,1}\|_{\LL^2}^p\big)  \leq C \ |t_2-t_1|^{(\frac{3}{4}-\varepsilon)p} \, \big[ 1+ \EE \|u_0\|_V^{(2\alpha +1)p}\big].
\end{equation}
The Minkowski inequality 	and  the upper estimates \eqref{3.2} and \eqref{3.8}  imply
\begin{align*}
\| T_{3,2}\|_{\LL^2} \leq & \int_{t_1}^{t_2} \|A^{\frac{1}{4}} S(t_2-s) A^{-\frac{1}{4}}\big(  |u(s)|^{2\alpha} u(s)\big)  \|_{\LL^2} ds \\
\leq & \; C(T) \int_{t_1}^{t_2} (t_2-s)^{-\frac{1}{4}} \|u(s)\|_V^{2\alpha +1} ds \leq C(T)  |t_2-t_1|^{\frac{3}{4}} \sup_{s\in [t_1,t_2]} \|u(s)\|_V^{2\alpha +1}.
\end{align*}

\noindent Then using once more \eqref{2.18} we obtain for $p\in [1, \infty)$, 
\begin{equation} 			\label{3.18}
\EE( \|T_{3,2}\|_{\LL^2}^p) \leq C(T,p) |t_2-t_1|^{\frac{3}{4} p}  \, \big[ 1+ \EE \|u_0\|_V^{(2\alpha +1)p}\big].
\end{equation}
A similar decomposition of the stochastic integral yields $T_4=T_{4,1}+T_{4,2}$, where
\[ 
 T_{4,1}=\! \int_0^{t_1} \!\! S(t_1-s) \big[ S(t_2-t_1)- \mbox{\rm Id }\big] G(u(s)) dW(s), \quad T_{4,2}=\! \int_{t_1}^{t_2} \!\! S(t_2-s) G(u(s)) dW(s).
 \]
The Burkholder-Davis-Gundy inequality, the growth condition \eqref{growthG_H}, \eqref{3.2} and \eqref{3.3} imply for $\epsilon
\in \big(0,\frac{1}{2}\big)$ and $p\in [1,\infty)$,
\begin{align}				\label{3.19}
\EE(\|T_{4,1}&\|_{\LL^2}^{2p}) \leq \,  C_p\,  \EE\Big( \Big| \int_0^{t_1} \| S(t_1-s)  \big[ S(t_2-t_1)- \mbox{\rm Id }\big] 
G(u(s)) \|_{\mathcal L}^2 \, \mbox{\rm Tr} Q\,  ds 
\Big|^p \Big)
\nonumber \\
\leq &\,  C_p \, \EE\Big( \Big| \int_0^{t_1} \| A^{\frac{1}{2}-\varepsilon} S(t_1-s) \|_{{\mathcal L}(\LL^2;\LL^2)}^2
\|A^{-(\frac{1}{2}-\varepsilon)} \big[ S(t_2-t_1)- \mbox{\rm Id }\big]  \|_{{\mathcal L}(\LL^2;\LL^2)}^2  \nonumber \\
&\qquad \times 
\| G(u(s))\|_{\mathcal L}^2 \; \mbox{\rm Tr}Q\; ds \Big|^p \Big) \nonumber \\
\leq & \, C_p\, (\mbox{\rm Tr} Q)^p \; |t_2-t_1|^{(1-2\varepsilon)p} \Big[ K_0^p + K_1^p \, \EE\Big( \sup_{s\in [0,t_1]} \|u(s)\|_H^2\Big)\Big]
\Big( \int_0^{t_1} (t_1-s)^{-1+2\varepsilon} ds
\Big)^p \nonumber \\
\leq & \, C(T,p, \mbox{\rm Tr} Q) \, |t_2-t_1|^{(1-2\varepsilon)p} \big[ 1+\EE\big(\|u_0\|_V^{2p}\big) \big],
\end{align} 
where the last upper estimate is deduced from \eqref{2.18}.

Finally, using once more the  Burkholder-Davis-Gundy inequality, $\sup_{t\in [0,T]} \|S(t)\|_{{\mathcal L}(\LL^2;\LL^2)} <\infty$, the
growth condition \eqref{growthG_H} and \eqref{2.18}, we obtain for $p\in [1,\infty)$
 \begin{align}			\label{3.20}
\EE(\|T_{4,2}\|_{\LL^2}^{2p}) \leq &\, C_p\, \EE\Big( \Big| \int_{t_1}^{t_2} \| S(t_2-s) G(u(s))\|_{\mathcal L}^2 \, \mbox{\rm Tr } Q\, ds \Big|^p \Big) \nonumber \\
\leq & \, C_p \, (\mbox{\rm Tr } Q)^p \; \EE\Big( \Big| \int_{t_1}^{t_2} \| S(t_2-s)\|_{{\mathcal L}(\LL^2;\LL^2)}^2 \big[ K_0 + K_1 \|u(s)\|_H^2\big] ds\Big|^p\Big) \nonumber \\
\leq & \, C_p \, (\mbox{\rm Tr } Q)^p \; 
|t_2-t_1|^{p} \big[ 1+\EE\big(\|u_0\|_V^{2p}\big) \big].
\end{align}    
The upper estiimates \eqref{ET1}-- \eqref{3.20} conclude the proof of \eqref{3.9}.
\medskip

 (ii) For $j=1, \cdots, N$ and $s\in [t_{j-1},t_j)$ we have  $\nabla u(t_j) - \nabla u(s)= \sum_{i=1}^4 T_i(s,j) $, where
\begin{align*}
T_1(s,j)=& \; \nabla S(t_j) u_0 - \nabla S(s) u_0, \\
T_2(s,j) = & -\int_0^{t_j} \nabla S(t_j-r) B(u(r), u(r)) dr + \int_0^s \nabla S(s-r) B(u(r),u(r)) dr, \\
T_3(s,j) =& -a \int_0^{t_j} \nabla S(t_j-r) \big( |u(r)|^{2\alpha} u(r) \big) dr + a \int_0^s \nabla S(s-r)  \big( |u(r)|^{2\alpha} u(r) \big)  dr, \\
T_4(s,j)=& \int_0^{t_j} S(t_j-r) \nabla G(u(r)) dW(r) - \int_0^s S(s-r) \nabla G(u(r)) dW(r).
\end{align*}
Using the upper estimates \eqref{3.2} and \eqref{3.3} we obtain
\[ \|T_1(s,j)\|_{\LL^2} = \| A^\delta S(s)\; A^{-\delta} \big[ S(t_j-s)- \mbox{\rm Id } \big] A^{\frac{1}{2}} u_0\|_{\LL^2} \leq C s^{-\delta} |t_j-s|^{\delta}
\|u_0\|_{V}
\]
for any  $\delta \in (0,1]$. Therefore, given  any $\delta\in (0,\frac{1}{2})$, we deduce
\begin{equation}				\label{3.22}
\sum_{j=1}^N \int_{t_{j-1}}^{t_j} \!\! \|T_1(s,j)\|_{\LL^2}^{2} ds \leq C \Big(\frac{T}{N}\Big)^{2\delta } \|u_0\|_V^{2} \int_0^T \!\! s^{-2\delta} ds = 
 C(T,\lambda)  \Big(\frac{T}{N}\Big)^{2\delta } \|u_0\|_V^{2}.
\end{equation}
As in the proof of (i), let $T_2(s,j) = - \big( T_{2,1}(s,j) + T_{2,2}(s,j)\big)$, where
\begin{align*}
 T_{2,1}(s,j) =&\;  \!  \int_0^s \!\! \nabla S(s-r) \big[ S(t_j-s)- \mbox{\rm Id } \big] B(u(r),u(r)) dr, \\
   T_{2,2}(s,j) =&\;  \! \int_s^{t_j} \!\! \nabla S(t_j-r) B(u(r),u(r)) dr.
\end{align*}
The Minkowski inequality and the upper estimates \eqref{3.2}, \eqref{3.3} and \eqref{3.5Bis} imply for $\delta \in \big( 0, \frac{1}{2})$  and
$\gamma \in (0, \frac{1}{2}-\delta)$ 
\begin{align}				\label{3.23-1}
\sum_{j=1}^N & \int_{t_{j-1}}^{t_j} \|T_{2,1}(s,j)\|_{\LL^2}^{2} ds  \nonumber \\
& \leq \,  \sum_{j=1}^N \int_{t_{j-1}}^{t_j}\!  ds \Big\{ \int_0^s \| A^{\frac{1}{2} +  \delta+\gamma} S(s-r) 
A^{-\gamma} \big[ S(t_j-s) - \mbox{ \rm Id } \big] A^{-\delta} B(u(r),u(r)) dr \|_{\LL^2} dr \Big\}^{2} 
\nonumber \\
&\leq \,  C \Big( \frac{T}{N}\Big)^{2\gamma }
\sum_{j=1}^N \int_{t_{j-1}}^{t_j}\!  ds    \Big\{ \int_0^s (s-r)^{-(\frac{1}{2} + \delta +\gamma)}
 \| A u(r)\|_{\LL^2}^{\frac{3}{4}-\delta}   \|u(r)\|_V^{\frac{5}{4}+\delta} dr \Big\}^{2} \nonumber \\
 &\leq   C \Big( \frac{T}{N}\Big)^{2\gamma }   
  \int_0^T ds  \Big(\int_0^s (s-r)^{- (\frac{1}{2} + \delta +\gamma)} dr \Big)
   \Big( \int_0^s (s-r)^{- (\frac{1}{2} + \delta +\gamma)} 
  \|Au(r)\|_{\LL^2}^{2 ( \frac{3}{4}-\delta)} dr \Big)   \nonumber \\
  &\qquad \times \sup_{r\in [0,T]}  \|u(r)\|_V^{2(\frac{5}{4}+\delta)}     
 \end{align}
 where in the last upper estimate, we have used the Cauchy-Schwarz inequality with respect to the measure 
 $(s-r)^{-\frac{1}{2}-\delta-\gamma} 1_{(0,s)}(r) dr$.
 
 Since $\int_r^T (s-r)^{-(\frac{1}{2}+\delta+\gamma)} ds \leq \int_0^T
  s^{-(\frac{1}{2}+\delta+\gamma)} ds=C(T,\delta,\gamma)$ for any $r\in [0,T)$,  and 
  $\int_0^s (s-r)^{- (\frac{1}{2} + \delta +\gamma)} dr \leq C(T,\delta,\gamma)$ for any $s\in [0,T)$, using the
  Fubini theorem, H\"older's  and Jensen's inequalities  with respect to $dP$ with conjugate exponents 
  $\frac{1}{\frac{1}{4}+\delta}$ and $\frac{1}{\frac{3}{4}-\delta}$ , we deduce 
 \begin{align*}  
 \EE & \sum_{j=1}^N  \int_{t_{j-1}}^{t_j} \|T_{2,1}(s,j)\|_{\LL^2}^{2} ds \leq  C \Big( \frac{T}{N}\Big)^{2\gamma } 
 C(T,\delta,\gamma)  \\ 
 &\qquad \qquad \times  \EE \Big( \sup_{r\in [0,T]}  \|u(r)\|_V^{2(\frac{5}{4}+\delta)}  \int_0^T  dr 
  \|Au(r)\|_{\LL^2}^{2 ( \frac{3}{4}-\delta)}  \int_r^T (s-r)^{- (\frac{1}{2} + \delta +\gamma)}  ds \Big) \\
 \leq & \;  C \Big( \frac{T}{N}\Big)^{2\gamma }  C(T,\delta,\gamma)^2 
 \Big\{ \EE\Big(  \sup_{r\in [0,T]}  \|u(r)\|_V^{\frac{2(5+4\delta) }{1+4\delta}}\Big)
    \Big\}^{\frac{1}{4}+\delta}
    \Big\{ \EE\Big(  \int_0^T  \|Au(r)\|_{\LL^2}^2 dr \Big) \Big\}^{\frac{3}{4}-\delta}. 
   \end{align*}
Let $\lambda \in (0,\frac{1}{2})$, $\delta=\frac{1-2\lambda}{4}\in (0,\frac{1}{4})$ and $\gamma \in (0, \frac{1}{2}-2\delta)$. 
Using \eqref{2.18}  we infer 
\begin{equation}	\label{3.23}
\EE\Big(  \sum_{j=1}^N  \int_{t_{j-1}}^{t_j} \|T_{2,1}(s,j)\|_{\LL^2}^{2} ds  \Big) 
  \leq C(T,\delta)  \Big( \frac{T}{N}\Big)^{2\lambda} \Big[ 1+\EE\Big(\|u_0\|_V^{\frac{6-2\lambda}{1-\lambda}}\Big)\Big] . 
  \end{equation}

  Using the Minkowski inequality, \eqref{3.2}, \eqref{3.5Bis} and H\"older's inequality for the measure $1_{[t_{j-1},t_j]}(s) ds$ with conjugate exponents 
  $p_1=\frac{2}{\frac{3}{4}-\delta}$ and $p_2=\frac{2}{\frac{5}{4}+\delta}$ we have $p_2(\frac{1}{2}+\delta)<1$ for $\delta \in (0,\frac{1}{4})$, 
  and deduce 
 \begin{align*} 			
 \sum_{j=1}^N  \!\int_{t_{j-1}}^{t_j} &\! \! \|T_{2,2}(s,j)\|_{\LL^2}^{2}ds \leq  \; C \sum_{j=1}^N \!\int_{t_{j-1}}^{t_j} \!\! ds 
 \Big\{ \int_s^{t_j} \big\| A^{\frac{1}{2}+\delta} S(t_j-r) \; A^{-\delta} B(u(r),u(r)) \big\|_{\LL^2} dr \Big\}^{2} \\ 
 \leq & \; C \sum_{j=1}^N \!\int_{t_{j-1}}^{t_j} \!\! ds 
 \Big\{ \int_s^{t_j} (t_j-r)^{-(\frac{1}{2}+\delta) } \|A u(r)\|_{\LL^2}^{\frac{3}{4}-\delta} \|u(r)\|_V^{\frac{5}{4} + \delta} dr \Big\}^{2} \\ 
 \leq & \; C  \sup_{r\in [0,T]}  \|u(r)\|_V^{\frac{5}{2} + 2\delta}
 \sum_{j=1}^N \int_{t_{j-1}}^{t_j} \! \Big( \int_s^{t_j} (t_j-r)^{-p_2(\frac{1}{2}+\delta)} dr\Big)^{\frac{2}{p_2}} 
 \Big( \int_s^{t_j}  \|A u(r)\|_{\LL^2}^{2}   dr \Big)^{\frac{2}{p_1}} ds \\
 \leq & \;  C \sup_{r\in [0,T]}  \|u(r)\|_V^{\frac{5}{2} + 2\delta} \Big( \frac{T}{N}\Big) ^{\frac{1}{4}-\delta} \sum_{j=1}^N \!\int_{t_{j-1}}^{t_j} \!\! ds 
 \Big( \int_{t_{j-1}}^{t_j}  \|A u(r)\|_{\LL^2}^{2}   dr \Big)^{\frac{2}{p_1}}.   
 \end{align*}
 The H\"older inequality for the counting measure on $\{ 1, ..., N\}$ 
 with conjugate exponents 
 $\frac{p_1}{2}= \frac{1}{\frac{3}{4}-\delta}$ and
 $\frac{1}{\frac{1}{4}+\delta}$ yields
 \begin{align*}
 \sum_{j=1}^N  \!\int_{t_{j-1}}^{t_j} &\! \! \! \|T_{2,2}(s,j)\|_{\LL^2}^{2}ds \leq  C(T,\delta) \Big( \frac{T}{N}\Big)^{\frac{5}{4}-\delta} 
  \sup_{r\in [0,T]}  \|u(r)\|_V^{\frac{5}{2} + 2\delta}  \Big\{ \sum_{j=1}^N \! \int_{t_{j-1}}^{t_j} \!\!\! \|Au(r)\|_{\LL_2}^2 dr \Big\}^{\frac{3}{4}-\delta}
N^{\frac{1}{4}+\delta} \\
  \leq & C(T,\delta) \, \Big(\frac{T}{N}\Big)^{1-2\delta} \!
   \sup_{r\in [0,T]}  \|u(r)\|_V^{\frac{5}{2} + 2\delta} \Big\{ \int_0^T   \|Au(r)\|_{\LL_2}^2 dr \Big\}^{\frac{3}{4}-\delta}.
 \end{align*} 
 H\"older's inequality with respect to $dP$ with conjugate exponents $\frac{1}{\frac{3}{4}-\delta}$ and
 $\frac{1}{\frac{1}{4}+\delta}$ implies
 \begin{align} 	\label{3.24}
 \EE \sum_{j=1}^N \int_{t_{j-1}}^{t_j}\!\!  \|T_{2,2}(s,j)\|_{\LL^2}^{2}  ds  \leq &\; C(T,\delta)   \Big(\frac{T}{N} \Big)^{ 1-2\delta}
  \Big\{ \EE\Big(  \sup_{r\in [0,T]}  \|u(r)\|_V^{ \frac{10+8\delta}{1 + 4\delta}} \Big) \Big\}^{\frac{1}{4}+\delta} \nonumber \\
 &\; \times \Big\{ \EE\Big( \int_0^T \|A u(r)\|_{\LL^2}^2  dr\Big) \Big\}^{\frac{3}{4}-\delta}.
 \end{align}
Let $\lambda \in (0,\frac{1}{2})$ and  $\delta=\frac{1-2\lambda}{4}\in (0,\frac{1}{4})$. 
The inequalities \eqref{3.23}, \eqref{3.24} and \eqref{2.18} imply 
\begin{equation} 			\label{3.25}
\EE  \sum_{j=1}^N \int_{t_{j-1}}^{t_j} \|T_2(s,j)\|_{\LL^2}^{2} ds \leq C(T,\lambda) 
 \Big( \frac{T}{N}\Big)^{2\lambda } \Big[ 1+\EE\big (\|u_0\|_V^{ \frac{6-2\lambda}{1-\lambda}}\Big)\Big].
\end{equation}
A similar decomposition yields $T_3(s,j) =-a\big(  T_{3,1}(s,j) + T_{3,2}(s,j) \big)$, where
\begin{align*}
T_{3,1}(s,j) =& \! \int_0^s \!\! \nabla S(s-r) \big[ S(t_j-s) - \mbox{\rm Id } \big] \big( |u(r)|^{2\alpha} u(r) \big) dr, \\
T_{3,2}(s,j) \!=& \int_s^{t_j}  \!\!\nabla S(t_j-r) \big( |u(r)|^{2\alpha} u(r) \big) dr.
\end{align*}
The Minkowski inequality and the upper estimates \eqref{3.2}, \eqref{3.3}, \eqref{3.8Bis} imply for $\delta \in (0, \frac{1}{2})$
and $\gamma \in (0, \frac{1}{2}-\delta)$, 
\begin{align*}				
 \|T_{3,1}(s,j)\|_{\LL^2}   \leq & \int_0^s \big\|A^{\frac{1}{2}+\delta +\gamma} S(s-r) \; A^{-\gamma} 
\big[ S(t_j-s) - \mbox{\rm Id } \big]\, A^{-\delta} \big( |u(r)|^{2\alpha} u(r)\big) \big\|_{\LL^2} dr \\
\leq & \; C (t_j-s)^\gamma \int_0^s (s-r)^{-(\frac{1}{2}+\delta +\gamma)} \|A u(r)\|_{\LL^2}^{\frac{3}{4}-\delta} 
\|u(r)\|_{V}^{2\alpha + \frac{1}{4} +\delta} dr.
\end{align*}
 Therefore, given $\delta\in \big(0,\frac{1}{2}\big)$ and $\gamma \in (0, \frac{1}{2}-\delta)$
 \begin{align*} 			
 \sum_{j=1}^N \int_{t_{j-1}}^{t_j}\ &\! \!  \|T_{3,1}(s,j)\|_{\LL^2}^{2}  ds \\
 \leq &\;  C \Big( \frac{T}{N}\Big)^{2\gamma} 
  \int_0^T  \Big\{ \int_0^s (s-r)^{-(\frac{1}{2}+\delta +\gamma)} \|A u(r)\|_{\LL^2}^{\frac{3}{4}-\delta} \|u(r)\|_{V}^{2\alpha + \frac{1}{4} +\delta}
 dr\Big\}^2 ds,
  \end{align*} 
  which is similar to \eqref{3.23-1} replacing the exponent $\frac{5}{4}+\delta $ of $\|u(r)\|_V$ by $2\alpha + \frac{1}{4} + \delta$. 
  Therefore, we deduce for $\delta \in (0, \frac{1}{4})$ 
  \begin{equation}  \label{3.26}
  \EE \sum_{j=1}^N  \int_{t_{j-1}}^{t_j} \|T_{3,1}(s,j)\|_{\LL^2}^{2} ds   
  \leq C(T,\delta) \Big( \frac{T}{N}\Big)^{1-4\delta} \Big[ 1+\EE\Big(\|u_0\|_V^{ \frac{16\alpha +2+8\delta }{1+4\delta}}\Big)\Big]. 
  \end{equation} 
The Minkowski  inequality, \eqref{3.2} and \eqref{3.8Bis} imply for $\delta \in (0,\frac{1}{4})$ 
 \begin{align*} 			
 \sum_{j=1}^N  \!\int_{t_{j-1}}^{t_j} &\! \! \|T_{3,2}(s,j)\|_{\LL^2}^{2}\, ds  \leq   \sum_{j=1}^N \!\int_{t_{j-1}}^{t_j} \!\! ds 
 \Big\{ \int_s^{t_j} \big\| A^{\frac{1}{2}+\delta} S(t_j-r) \; A^{-\delta} \big( |u(r)|^{2\alpha} u(r) \big) \big\|_{\LL^2} dr \Big\}^{2} \\ 
  \leq & \; C \sum_{j=1}^N \!\int_{t_{j-1}}^{t_j} \!\! ds 
 \Big\{ \int_s^{t_j} (t_j-r)^{-(\frac{1}{2}+\delta) } \|A u(r)\|_{\LL^2}^{\frac{3}{4}-\delta} \|u(r)\|_V^{2\alpha + \frac{1}{4} + \delta} dr \Big\}^{2} 
  \end{align*}
 The arguments  for proving  \eqref{3.24} imply 
 \begin{align} 	\label{3.27}
 \EE  \sum_{j=1}^N \int_{t_{j-1}}^{t_j}\!\!  \|T_{3,2}(s,j)\|_{\LL^2}^{2}  ds \leq &\, C(T,\delta ) 
  \Big( \frac{T}{N} \Big)^{1-2\delta} 
 \Big\{ \EE\Big( \sup_{r\in [0,T]} \|u(r)\|_V^{ \frac{16\alpha +2+8\delta }{1+4\delta}   } \Big) \Big\}^{\frac{1}{4}+\delta} \nonumber \\
&\times  \Big\{ \EE\Big( \int_0^T \|A u(r)\|_{\LL^2}^2  dr\Big) \Big\}^{\frac{3}{4}-\delta}.
 \end{align}
The inequalities \eqref{3.26}, \eqref{3.27} and \eqref{2.18} imply that for $\lambda \in (0, \frac{1}{2})$ and $\delta = \frac{1-2\lambda}{4} 
\in \big( 0, \frac{1}{4}\big)$, 
\begin{equation} 			\label{3.28}
\EE \sum_{j=1}^N \int_{t_{j-1}}^{t_j} \|T_3(s,j)\|_{\LL^2}^{2} ds  \leq C(T,a,\lambda) 
 \Big( \frac{T}{N}\Big)^{2\lambda} \Big[ 1+\EE\Big(\|u_0\|_V^{p(\lambda)}\Big) \Big].
\end{equation}

Finally, the stochastic integral can be decomposed as follows: $T_4(s,j) = T_{4,1}(s,j) + T_{4,2}(s,j)$, where
\[ T_{4,1}(s,j)=\! \! \int_0^s \!\!\! S(s-r) \big[ S(t_j -s) - \mbox{\rm Id } \big] \nabla G(u(r)) dW(r), \; 
T_{4,2}(s,j)=\!\!\int_s^{t_j} \!\!\!S(t_j-r) \nabla G(u(r)) dW(r).\]
The 
$L^2(\Omega)$-isometry,  \eqref{3.2}, \eqref{3.3} and the growth condition  \eqref{growthG_V} 
imply for $\delta \in \big(0,\frac{1}{2} \big)$ 
\begin{align}			\label{3.29}
\EE
& \sum_{j=1}^N \int_{t_{j-1}}^{t_j} \!\! \! \| T_{4,1}(s,j)\|_{\LL^2}^{2} ds   \nonumber  
  \\
&\leq  \EE \sum_{j=1}^N 
\int_{t_{j-1}}^{t_j} \!\ \! \int_0^s \! \!  \! \big\|S(s-r) 
\big[ S(t_j -s) - \mbox{\rm Id } \big] A^{\frac{1}{2}} G(u(r)) \big\|_{\mathcal L}^2 \, \mbox{\rm Tr} Q \, dr  ds  \nonumber \\
&\leq \EE\,   \sum_{j=1}^N \int_{t_{j-1}}^{t_j} \! ds\!   \int_0^s \big\| A^{\frac{1}{2}-\delta} S(s-r)\big\|_{{\mathcal L}(\LL^2 ; \LL^2)}^2 
\big\| A^{-(\frac{1}{2}-\delta)} \big[ S(t_j -s) - \mbox{\rm Id } \big] \big\|_{{\mathcal L}(\LL^2 ; \LL^2)}^2 \nonumber \\
&\qquad\qquad\qquad  \times \|G(u(r))\|_{\widetilde{\mathcal L}}^2 \, \mbox{\rm Tr} Q\,  dr   \nonumber \\
&\leq  \mbox{\rm Tr} Q \;   \EE \int_0^T \! ds \! \int_0^s (s-r)^{-1+2\delta} (t_j-s)^{1-2\delta} \big[ \tilde{K}_0 
+ \tilde{K}_1 \|u(r)\|_V^2 \big] dr
\nonumber \\
&\leq  \mbox{\rm Tr} Q \;  \Big[ \tilde{K}_0 + \tilde{K}_1 \; 
\EE\Big( \sup_{r\in [0,T]} \|u(r)\|_V^{2} \Big) \Big] \, \Big( \frac{T}{N} \Big)^{1-2\delta}
 \int_0^T   s^{2\delta } ds \nonumber \\
 &\leq C(T,\mbox{\rm Tr } Q, \delta) 
\Big( \frac{T}{N} \Big)^{1-2\delta} \big[ 1+\EE(\|u_0\|_V^{2}) \big], 
\end{align} 
Finally, the $L^2(\Omega)$-isometry, $\sup_r  \|S(r)\|_{{\mathcal L}(\LL^2 ; \LL^2)}$ and 
the growth condition \eqref{growthG_V}  and \eqref{2.18} imply  
\begin{align}				\label{3.30}
\EE\Big( \sum_{j=1}^N 
 \!\int_{t_{j-1}}^{t_j} \!\! \! \| T_{4,2}(s,j)\|_{\LL^2}^{2} ds  \Big) 
&\leq \EE \sum_{j=1}^N \int_{t_{j-1}}^{t_j}  \int_s^{t_j} 
\| S(t_j-r)\|_{{\mathcal L}(\LL^2 ; \LL^2)}^2 \|  G(u(r))\|_{\widetilde{\mathcal L}}^2 \, \mbox{\rm Tr } Q\, dr ds  \nonumber \\
&\leq \mbox{\rm Tr Q} \; \EE \sum_{j=1}^N \int_{t_{j-1}}^{t_j} ds  \int_s^{t_j}  \big[ \tilde{K}_0 + \tilde{K}_1 \|u(r)\|_V^2 \big] dr\nonumber \\
&\leq  C(T, \mbox{\rm Tr } Q) \,   \frac{T}{N}  \, \big[ 1+\EE(\|u_0\|_V^{2}) \big].
\end{align}
 For $\alpha \in [1,\frac{3}{2}]$ and $\lambda \in (0,\frac{1}{2})$, $2<\frac{6-2\lambda}{1-\lambda} < p(\lambda):= 
\frac{2+8\alpha-2\lambda}{1-\lambda}$. 
Therefore,  the upper estimates \eqref{3.22}, \eqref{3.25}, \eqref{3.28}--\eqref{3.30} imply for  $\lambda \in (0,\frac{1}{2})$ 
\[ 
\EE  \sum_{j=1}^N \int_{t_{j-1}}^{t_j}\!\!   \| \nabla (u(s) - u(t_j))\|_{\LL^2}^{2} ds \leq C(T,a,\mbox{\rm Tr } Q, \lambda) 
\Big( \frac{T}{N} \Big)^{2\lambda} 
\Big[ 1+ \EE\Big(\|u_0\|_V^{p(\lambda)} \Big)\Big].
\] 
Small changes in the proof of this upper estimate prove that under similar assumptions 
\[ 
\EE  \sum_{j=1}^N \int_{t_{j-1}}^{t_j}\!\!   \| \nabla (u(s) - u(t_{j-1}))\|_{\LL^2}^{2} ds  \leq C(T, \mbox{\rm Tr} Q,\lambda)
 \Big( \frac{T}{N} \Big)^{2\lambda} 
\Big[ 1+ \EE\Big(\|u_0\|_V^{p(\lambda)  } \Big) \Big].
\] 
This completes the proof of \eqref{3.10}. 
\end{proof} 
\begin{remark} 
Note that the above proof shows that when time increments of the gradient of the solution are dealt with, due to the term containing the
initial condition, one cannot obtain moments of $\EE( \|u(t)-u(s)\|_V^{2})$ uniformly in $s,t$ with $0\leq s<t\leq T$. Furthermore, in order
to obtain the "optimal" time regularity, that is almost $\frac{1}{2}$, we also need a time integral.
\end{remark}

\section{Well-posedness and moment estimates of the  implicit time  Euler scheme} \label{sEuler}
We  first prove the existence of the fully time implicit time Euler scheme. Fix $N\in \{1,2, ...\}$, let $h:=\frac{T}{N}$ denote the time mesh,
 and for $j=0, 1, ..., N$ set $t_j:=j \frac{T}{N}$. 

The fully implicit time Euler scheme $\{ u^k ; k=0, 1, ...,N\}$ is defined by $u^0=u_0$ and for $\varphi \in V$
\begin{align}	\label{4.1}
\Big( u^k-u^{k-1} &+ h \nu A u^k + h B\big( u^k,u^k\big) + h\, a\,  |u^k|^{2\alpha} u_N(t_k)  , \varphi\Big)  \nonumber \\
&= \big( G(u^{k-1}) [W(t_k)-W(t_{k-1})]\, , \, \varphi), \qquad k=1,2, ..., N.
\end{align}
 Set $\Delta_j W:= W(t_j)-W(t_{j-1})$, $j=1, ...,N$.
\medskip

The following proposition states the existence and uniqueness of the sequence $\{ u^k\}_{k=0, ...,N}$ and provides moment estimates which do not
depend on $N$.
\begin{prop} 		\label{prop_uk}		
Let $\alpha \in [1,\frac{3}{2}]$ and Condition {\bf (G)} be satisfied. The time fully  implicit scheme \eqref{4.1} has a 
 solution $\{u^k\}_{k=1, ...,N} \in V\cap \HH^2$ 
  Furthermore,
\begin{align}	\label{mom_uk_V}
\sup_{N\geq 1} \EE\Big( \max_{k=0, ...,N}& \|u^k\|_{V}^2 + \frac{T}{N} \sum_{k=1}^N \|A u^k\|_{\LL^2}^2\nonumber \\
&+ \frac{T}{N}\sum_{k=1}^N \big[  \|u^k\|_{\LL^{2\alpha+2}}^{2\alpha+2} + \| |u^k|^\alpha \nabla u^k\|_{\LL^2}^2\big] \Big) <\infty.
\end{align}
\end{prop}
\begin{proof} The proof is divided in two steps.\\
{\bf Step 1: Existence of the scheme} 
We first prove that for fixed $N\geq 1$ \eqref{4.1} has a 
 solution in $V\cap \LL^{2\alpha +2}$. 
For technical reasons we consider a Galerkin approximation.
As in Section \ref{s-gwp}   let $\{e_l\}_l$ denote an orthonormal basis of $H$ made of elements of $\HH^2$ which are orthogonal in $V$. 
Since $\alpha \in [1, \frac{3}{2}]$, the Gagliardo Nirenberg inequality implies that $\HH^1 \subset \LL^{2\alpha +2}$.\\
 For $m=1,2, ...$ let
 $V_m=\mbox{ \rm span }(e_1, ..., e_m) \subset \HH^2$  and let $P_m:V\to V_m$ denote the projection from 
$V$ to $V_m$. 
 In order to find a solution to  \eqref{4.1} we project this equation on $V_m$, that
is we define by induction a sequence $\{u^k(m)\}_{k=0, ...,N} \in V_m$ such that $u^0(m)= P_m(u_0)$, and for $k=1, ...,N$ and $\varphi \in V_m$
\begin{align}		\label{def-ukm}
\big( u^k(m)-u^{k-1}(m), \varphi\big)  &+ h \Big[ \nu \big( \nabla u^k(m), \nabla \varphi)  + \big\langle B\big( u^k(m), u^k(m)\big), \varphi \big\rangle 
\nonumber  \\
&+  a\,  \big( |u^k(m)|^{2\alpha} u^k(m)  , \varphi\big)\Big]  
= \big( G(u^{k-1}(m)) \Delta_kW \, , \, \varphi\big).
\end{align}
For almost every $\omega$ set $R(0,\omega):= \|u_0(\omega)\|_{\LL^2}$. Fix  $k=1, ...,N$ and suppose that  for $j=0, ..., k-1$ the ${\mathcal F}_{t_j}$-
measurable random variables 
$u^j(m)$ have been defined, and that  
\[ R(j,\omega):=\sup_{m\geq 1}
\|u^j(m,\omega)\|_{\LL^2}<\infty\quad \mbox{\rm for almost every } \omega.\]
 We prove that $u^k(m)$ exists and satisfies a.s. $ \sup_{m\geq 1} \|u^k(m,\omega)\|_{\LL^2}<\infty$. 
The argument is based on the following result \cite[Cor 1.1]{GirRav}  page 279, which can be deduced from Brouwer's theorem.
\begin{prop} \label{G-R}  Let $H$ be a Hilbert space of finite dimension, $(.,.)_H$ denote its inner product, and $\Phi:H\to H$
be continuous such that for some $\mu>0$,
\[ \big( \Phi(f),f\big)_H \geq 0, \quad \mbox{\rm for all}\;  f\in H \; \mbox{\rm with } \|f\|_H=\mu.
\]
Then there exists $f\in H$ such that  $\Phi(f)=0$ and $\|f\|_H\leq \mu$. 
\end{prop} 
For $\omega \in \Omega$ let $\Phi^k_{m,\omega}:V_m\to V_m$ be defined for $f\in V_m$  as the solution of 
\begin{align*}
\big(\Phi^k_{m,\omega}(f) , \varphi\big) = &\; \big( f-u^{k-1}(m,\omega), \varphi\big)  + h \Big[ \nu \big( \nabla f,\nabla \varphi\big) 
+ \big\langle P_m B(f,f),\varphi
\big\rangle  +a \big( P_m(|f|^{2\alpha} f), \varphi\big)\Big] \\
&- \big( P_m G(u^{k-1}(m,\omega) \Delta_kW(\omega), \varphi \big), \qquad \forall \varphi \in V_m.
\end{align*}
Then
\begin{align*}  \big( \Phi^k_{m,\omega}(f),f\big) = &\; \|f\|_{\LL^2}^2 - \big( u^{k-1}(m,\omega), f\big) + h\nu \|\nabla f\|_{\LL^2}^2 
+ h\,a \|f\|_{\LL^{2\alpha +2}}^{2\alpha +2} \\
& -\big( G(u^{k-1}(m,\omega)\big)  \Delta_kW(\omega),f\big). 
\end{align*}
The Young inequality implies $\big|\big( u^{k-1}(m,\omega), f\big) \leq \frac{1}{2} \|f\|_{\LL^2}^2 + \frac{1}{2}  \|u^{k-1}(m,\omega)\|_{\LL^2}^2$ 
and the growth condition  \eqref{growthG_H} implies  
\begin{align*}
 \big| \big( G(u^{k-1}(m,\omega) \Delta_k W (\omega),f\big)\big| \leq &\; \big\|G\big(u^{k-1}(m,\omega)\big)\|_{\mathcal L} 
\|\Delta_kW(\omega)\|_{K}\, \|f\|_{\LL^2}\\
\leq &\; \frac{1}{4} \|f\|_{\LL^2}^2 + \big[ K_0 + K_1 \|u^{k-1}(m,\omega)\|_{\LL^2}^2 \big]  \,   \|\Delta_k W(\omega)\|^2_{K}. 
\end{align*}
Hence 
\[ \big( \Phi^k_{m,\omega}(f) ,f\big) \geq \frac{1}{4} \|f\|_{\LL^2}^2 -  \frac{1}{2} \|u^{k-1}(m,\omega)\|_{\LL^2}^2   - \big[ K_0 + K_1 
\|u^{k-1}(m,\omega)\|_{\LL^2}^2 \big] \|\Delta_k W(\omega)\|^2_K  \geq 0 \] 
 if 
\[ \|f\|_{\LL^2}^2 = R^2(k,\omega) := 4\Big[ K_0 \|\Delta_k W(\omega)\|^2_{K}
 + R^2(k-1,\omega)\Big( \frac{1}{2} +K_1 \|\Delta_k W(\omega)\|^2_K
\Big) \Big]. \]
Proposition \ref{G-R} implies the existence of  
$u^k(m,\omega) \in V_m$ such that $\Phi^k_{m,\omega}\big(u^k(m,\omega)\big) = 0$, 
and $\|u^k(m,\omega) \|^2_{\LL^2} \leq R^2(k,\omega)$; note that this element $u^k(m,\omega)$ need not be unique. 
Furthermore, the random variable $u^k(m)$ is ${\mathcal F}_{t_k}$-measurable. 

The definition of $u^k(m)$ implies that it is a solution to \eqref{def-ukm}. Taking $\varphi = u^k(m)$ in \eqref{def-ukm} and using the Young
inequality, we obtain 
\begin{align*}
\| u^k(m)&\|_{\LL^2}^2 + h\, \nu  \|\nabla u^k(m)\|_{\LL^2}^2 + h\, a \|u^k(m)\|_{\LL^{2\alpha +2}}^{2\alpha +2} \\
&=  \big( u^{k-1}(m), u^k(m)\big) +  \big( G(u^{k-1}(m) \Delta_k W, u^k(m)\big)
\\
&\leq \frac{1}{4} \|u^k(m)\|_{\LL^2}^2 + \|u^{k-1}(m)\|_{\LL^2}^2 + \frac{1}{4} \|u^k(m)\|_{\LL^2}^2 + \big[ K_0+K_1 \|u^{k-1}(m)\|_{\LL^2}^2\big]
\|\Delta_k W\|_{K}^2.
\end{align*} 
Hence a.s. 
\begin{align*}
 \sup_{m\geq 1} \Big[ \frac{1}{2} \|u^k(m,\omega)\|_{\LL^2}^2 &+ h\, \nu \|\nabla u^k(m,\omega)\|_{\LL^2}^2 +
h\, a \| u^k(m,\omega)\|_{\LL^{2\alpha +2}}^{2\alpha +2} \Big] \\
&\leq \; R^2(k-1,\omega) \big[ 1+K_1 \|\Delta_k W(\omega)\|_{K}^2 \big] + K_0   \|\Delta_k W(\omega)\|_K^2 ,
\end{align*}  
Therefore, for fixed $k$ and almost every $\omega$, the sequence $\{u^k(m,\omega)\}_m$ is bounded in $V\cap \LL^{2\alpha+2}$; it has a
subsequence (still denoted $\{u^k(m,\omega)\}_m$) which converges weakly in $V\cap \LL^{2\alpha +2}$ to $\phi_k(\omega)$. The random variable
$\phi_k$ is ${\mathcal F}_{t_k}$-measurable. 

Since $D$ is bounded, the embedding of  $V$  in $H$ is compact; hence the subsequence $\{u^k(m,\omega)\}_m$ converges strongly to 
$\phi_k(\omega)$ in $\LL^2$.
 
  Then by definition $u^0(m)$ converges strongly to $u_0$. 
  We next prove by induction on $k$ that $\phi^k$ solves \eqref{4.1}. 
  Fix a positive integer $m_0$ and consider the equation \eqref{def-ukm} for $k=1, ..., N$, 
   $\varphi \in V_{m_0}$,  and $m\geq m_0$.
  As $m\to \infty$ we have a.s. 
  \[ \big(u^k(m)-u^{k-1}(m) , \varphi) \to \big(\phi^k-\phi^{k-1} , \varphi) .\]
  Furthermore, the antisymmetry of $B$ \eqref{B} and the Gagliardo-Nirenberg inequality $\|g\|_{\LL^4} \leq C \|\nabla g\|_{\LL^2}^{\frac{3}{4}}
  \|g\|_{\LL^2}^{\frac{1}{4}}$ yield a.s. 
  \begin{align*}
  \big| \big\langle B\big(u^k(m),u^k(m)\big)& - B(\phi^k,\phi^k), \varphi \big\rangle \big| \\
  \leq &\; \big| \big\langle B\big(u^k(m)-\phi^k, \varphi \big), u^k(m)  
  \big\rangle \big| +\big| \big\langle B\big(\phi^k, \varphi \big), u^k(m)  -\phi^k \big\rangle \big| \\
  \leq & \; \|\nabla \varphi\|_{\LL^2} \|u^k(m)-\phi^k\|_{\LL^4} \, \big[ \|u^k(m)\|_{\LL^4} + \| \phi^k\|_{\LL^4}\big] \\
  \leq & C \; \|\varphi\|_{\LL^2}  \big[ \max_{m}\|u^k(m)\|_V^{\frac{7}{4}}  + \| \phi^k\|_V^{\frac{7}{4}} \big] \| u^k(m)-\phi^k\|_{\LL^2}^{\frac{1}{4}} \to 0
  \end{align*} 
 as $m\to \infty$. The inequality \eqref{1.9} implies
\begin{align*} 
\big| \big(|u^k(m)|^{2\alpha} u^k(m) -& |\phi^k|^{2\alpha} \phi^k, \varphi\big) \big| \leq \; C \int |u^k(m)-\phi^k| \big( |u^k(m)|^{2\alpha} + |\phi^k|^{2\alpha}\big)
\, |\varphi |\,  dx \\ 
\leq &\, C  \|\varphi\|_{\LL^\infty} \|u^k(m)-\phi^k\|_{\LL^2} \big( \|u^k(m)\|_{L^{4\alpha}}^{4\alpha}  + \|\phi^k\|_{\LL^{4\alpha}}^{4\alpha} \big) \\
\leq &\; C \|\varphi\|_{\HH^2} \big( \max_m  \|u^k(m)\|_{V}^{4\alpha}  + \|\phi^k\|_{V}^{4\alpha} \big) \|u^k(m)-\phi^k\|_{\LL^2} \to 0
\end{align*}
as $m\to \infty$. Note that the last upper estimate follows from the inclusion $\HH^1\subset \LL^p$ for $p\in [2,6]$ and $\alpha \in [1,\frac{3}{2}]$. 
Finally, the Cauchy-Schwarz inequality and the Lipschitz condition \eqref{LipG} imply
\begin{align*}
\big| \big(  G\big( u^{k-1}(m)\big) \Delta_kW, \varphi\big) - &\big( G\big(  \phi^{k-1} \big) \Delta_kW, \varphi\big) \big| 
\leq \|\varphi\|_{\LL^2}
 \| G(u^{k-1}(m)- G(\phi^{k-1})  \|_{\mathcal L} \| \Delta_kW\|_K \\
 &\leq  \sqrt{L}\, \|\varphi\|_{\LL^2}\, \|u^{k-1}(m)-\phi^{k-1}\|_{\LL^2} \, \| \Delta_kW\|_K  \to 0
\end{align*}
as $m\to \infty$. Therefore, letting $m\to \infty$ in \eqref{def-ukm}, we deduce 
\begin{align*}
\Big( \phi^k -\phi^{k-1}  &+ h \nu A \phi^k + h B\big( \phi^k ,\phi^k \big) + h\, a\,  |\phi^k |^{2\alpha} \phi^k  , \varphi\Big) 
= \big( G(\phi^{k-1} ) \Delta_k W \, , \, \varphi), \quad \forall \varphi \in V_{m_0}. 
\end{align*} 
Since $\cup_{m_0} V_{m_0}$ is dense in $V$, we deduce that $\phi^k$ is a solution to \eqref{4.1}. 

\noindent {\bf Step 2: Moment estimates} 
We next prove \eqref{mom_uk_V} for any solution $\{ u^k\}_{k=0, ...,N}$ to \eqref{4.1}. We  first study the $\LL^2$-norm of the sequence. 
 Write \eqref{4.1} with $\varphi=u^k$ and use the identity $(f,f-g) = \frac{1}{2} \big[ \|f\|_{\LL^2 }- \|g\|_{\LL^2}^2 
+ \| f-g\|_{\LL^2}^2\big]$. Using the Cauchy-Schwarz and Young inequalities, and the growth condition \eqref{growthG_H}, this yields for $k=1, ..., N$  
\begin{align*}
\frac{1}{2}  \|u^k\|_{\LL^2}^2 &- \frac{1}{2}  \|u^{k-1} \|_{\LL^2}^2 + \frac{1}{2}  \|u^k-u^{k-1}\|_{\LL^2}^2 +h\nu \|\nabla u^k\|_{\LL^2}^2
+ha \|u^k\|_{\LL^{2\alpha +2}}^{2\alpha +2}  \\
& =
 \big( G(u^{k-1}) \Delta_k W\, , u^k-u^{k-1}\big) + \big( G(u^{k-1}) \Delta_k W\, , u^{k-1}\big) \\
 &\leq \frac{1}{2} \|u^k-u^{k-1}\|_{\LL^2}^2  + \frac{1}{2} \big[ K_0 + K_1\|u^{k-1}\|_{\LL^2}^2\big]  \|\Delta_k W\|_{K}^2 +
 \big( G(u^{k-1}) \Delta_k W\, , u^{k-1}\big). 
\end{align*}
For any $K=1, ..., N$, adding the above inequalities for $k=1, ...,K$ we deduce
\begin{align}		\label{majscheme1}
&\|u^K\|_{\LL^2}^2 +2 h\nu \sum_{k=1}^K \|\nabla u^k\|_{\LL^2}^2
+2ha \sum_{k=1}^K \|u^k\|_{\LL^{2\alpha +2}}^{2\alpha +2} \nonumber  \\
&\quad  \leq  \|u_0\|_{\LL^2}^2 + \sum_{k=1}^K \big[ K_0 + K_1\|u^{k-1}\|_{\LL^2}^2\big]  \|\Delta_k W\|_{K}^2  
+2  \sum_{k=1}^K  \big( G(u^{k-1}) \Delta_k W\, , u^{k-1}\big). 
\end{align}
Therefore,
\begin{align*}
\EE\Big( \max_{1\leq K\leq N}& \|u^K\|_{\LL^2}^2\Big) + 2h\EE\Big( \sum_{k=1}^N \big[ \nu \|\nabla u^k\|_{\LL^2}^2 
+a\|u^k\|_{\LL^{2\alpha +2}}^{2\alpha +2}\big] \Big) \\
\leq & \; 2 \EE\Big( \max_{1\leq K\leq N} \Big[ \|u^K\|_{\LL^2}^2 +  2h \sum_{k=1}^K \big( \nu \|\nabla u^k\|_{\LL^2}^2 
+a\|u^k\|_{\LL^{2\alpha +2}}^{2\alpha +2}\big)\Big]
\Big) \\
\leq  &\; 2 \EE(\|u_0\|_{\LL^2}^2) + 2 h {\rm Tr} (Q) \sum_{k=0}^{N-1} \big[ K_0+K_1\EE(\|u^k\|_{\LL^2}^2)  \big] \\
& +  4 \EE\Big( \max_{1\leq K\leq N} \sum_{k=1}^K  \big( G(u^{k-1}) \Delta_k W\, , u^{k-1}\big) \Big).
\end{align*}
The Davis and then Young inequalities imply
\begin{align*}
\EE\Big( \max_{1\leq K\leq N}& \sum_{k=1}^K  \big( G(u^{k-1}) \Delta_k W\, , u^{k-1}\big) \Big) \leq 3 \EE\Big( \Big\{ \sum_{k=0}^{N-1} \|u^{k}\|_{\LL^2}^2
\big[ K_0+K_1 \|u^{k}\|_{\LL^2}^2\big] h {\rm Tr} Q\Big\}^{\frac{1}{2}} \Big) \\
&\leq \frac{1}{4} \EE\Big( \max_{0\leq k\leq N-1} \|u^k\|_{\LL^2}^2\Big) + 9 \EE\Big( h{\rm Tr} Q \sum_{k=0}^{N-1} \big[  K_0+K_1 \|u^{k}\|_{\LL^2}^2\big]
\Big).
\end{align*}
Hence we deduce
\begin{align}  	\label{partiel1}
\frac{1}{2} \EE\Big( \max_{1\leq K\leq N} \|u^K\|_{\LL^2}^2\Big) &+ 2h\EE\Big( \sum_{k=1}^N \big[ \nu \|\nabla u^k\|_{\LL^2}^2 
+a\|u^k\|_{\LL^{2\alpha +2}}^{2\alpha +2}\big] \Big)  \nonumber \\
\leq  & \; 2 \EE(\|u_0\|_{\LL^2}^2) + 74  T K_0 {\rm Tr} Q + 74 K_1 {\rm Tr} Q \sum_{k=0}^{N-1} h \EE(\|u^k\|_{\LL^2}^2).
\end{align} 
Neglecting the sum in the left hand side and using the discrete Gronwall lemma, we obtain
\[  \sup_{N\geq 1} \EE\Big( \max_{1\leq K\leq N} \|u^K\|_{\LL^2}^2\Big)  \leq C(T, {\rm Tr} Q, \|u_0\|_{\LL^2}^2, K_0, K_1).\]
Plugging this upper estimate in \eqref{partiel1}, we obtain
\begin{align*}	\label{mom_uk}
\sup_{N\geq 1} \EE\Big( \max_{k=0, ...,N}& \|u^k\|_{\LL^2}^2 + \frac{T}{N} \sum_{k=1}^N  \big[ \nu \|\nabla u^k\|_{\LL^2}^2
+ a \|u^k\|_{\LL^{2\alpha+2}}^{2\alpha+2} \big] \Big) <\infty.
\end{align*} 
A similar argument  with $\varphi = A u^k$, integrating by parts, and using Lemma \ref{upper_nablaB} and inequality
\eqref{*}  yields
\[ \sup_{N\geq 1}\EE\Big( \max_{1\leq K\leq N} \|\nabla u^K\|_{\LL^2}^2 + \frac{T}{N} \sum_{k=1}^N \big[ \|A u^k\|_{\LL^2}^2 
 + \| |u^k|^{\alpha } \nabla u^k\|_{\LL^2}^2 \big] \Big)  = C_2(\alpha) <\infty.\]
This completes the proof of the proposition.
\end{proof}

\section{Strong convergence of the implicit time Euler scheme}		\label{s-convergence}
Let $u$ be the solution to \eqref{3D-NS} and $\{ u^j:= u_N(t_j)\}_{j=0, ...,N}$ solve the fully implicit time Euler scheme defined in \eqref{4.1}. 
Let $e_j:= u(t_j)-u^j$. Using 
\eqref{3D-NS} and \eqref{4.1}, we deduce $e_0=0$ and for $j=1, ..., N$ and $\varphi \in V$
\begin{align}	\label{5.1}
\big( &e_j-e_{j-1}\, , \, \varphi \big) + \nu \int_{t_{j-1}}^{t_j}  \big( \nabla u(s) - \nabla u^j\, ,\,  \nabla \varphi\big)  ds 
+ \int_{t_{j-1}}^{t_j}  \big\langle B(u(s),u(s)) - B(u^j,u^j)\, , \, \varphi\big\rangle ds \nonumber \\
&+ a  \int_{t_{j-1}}^{t_j} \big( |u(s)|^{2\alpha}u(s) -  |u^j|^{2\alpha}u^j\, ,  \, \varphi \big) ds =  
 \int_{t_{j-1}}^{t_j} \big( [G(u(s))-G(u^{j-1}) ] dW(s) \, , 
\, \varphi\big).  
\end{align}
Note that since $\alpha \in [1,\frac{3}{2}]$ and $\HH^1\subset \LL^p$ for $p\in [2,6]$, H\"older's inequality with exponents $2,3$ and $6$ implies that
the space integral defining the inner product  $\big( |u(s)|^{2\alpha}u(s) -  |u^j|^{2\alpha}u^j\, ,  \, \varphi \big) $ is converging for $u(s), u^j, \varphi \in V$.
The following convergence theorem is one of the main results of this paper.
\begin{theorem}	\label{strong-fully-alpha>1}
Suppose that   condition {\bf (G)} holds. Let $\alpha \in [1,\frac{3}{2}]$; when $\alpha=1$, suppose that $4\nu a (1\wedge \kappa) >1$,
 where $\kappa>0$ is the constant defined in inequality \eqref{1.10}.    \\
Fix $\lambda \in (0,\frac{1}{2} )$ and set $p(\lambda)=\frac{2+8\alpha-2\lambda}{1-\lambda}$. 
Let  $u_0\in L^{p(\lambda)}(\Omega ; V)$, $u$ be the solution to \eqref{3D-NS} and $\{u^j\}_{j=0, ...,N}$
solve the fully implicit scheme \eqref{4.1}. Then there exists a  positive constant  $C:=C(\nu, \alpha, a, \kappa, {\rm Tr}\, Q)$
 independent of $N$ such that  for $N$ large enough
\begin{align} 		\label{strong-full->1}
\EE\Big( \max_{1\leq j\leq N} \|u(t_j)-u^j\|_{\LL^2}^2 &+ \frac{T}{N} \sum_{j=1}^N \|\nabla[ u(t_j) - u^j] \|_{\LL^2}^2 \Big) \nonumber \\
&\leq C \, \Big( \frac{T}{N}\Big)^{2\lambda}  \Big[ 1+ \EE\Big( \|u_0\|_V^{p(\lambda)}\Big)\Big].
\end{align}
\end{theorem}
\begin{remark}
Note that the various parameters of the model $\nu, \alpha, a, {\rm Tr}\,(Q)$ only appear in the multiplicative constant $C$ in the right hand side
of \eqref{strong-full->1}, but not in the exponent $\lambda$ which can be chosen arbitrarily close to $\frac{1}{2}$  if $u_0\in V$ is deterministic, or
if $u_0$ is a $V$-valued Gaussian random variable independent of $W$.
\end{remark}  
\noindent {\it Proof of Theorem \ref{strong-fully-alpha>1}}\\
\noindent  (i) We  first suppose that $\alpha \in (1, \frac{3}{2}]$.

Using the identity \eqref{5.1} with $\varphi = e_j$, the equality $(f,f-g)=\frac{1}{2}\big[ \|f\|_{\LL^2}^2 - \|g\|_{\LL^2}^2 + \|f-g\|_{\LL^2}^2\big]$ 
and the estimate  \eqref{1.27}, we deduce that for some $\kappa >0$ we have for $j=1, ..., N$
\begin{align}	\label{5.2}
\frac{1}{2} \big( \|e_j\|_{\LL^2}^2 - \|e_{j-1}\|_{\LL^2}^2\big) &+ \frac{1}{2} \|e_j-e_{j-1}\|_{\LL^2}^2 + \nu h \|\nabla e_j\|_{\LL^2}^2  \nonumber \\
& + a\kappa h \| |u(t_j)|^\alpha e_j\|_{\LL^2}^2 + a\kappa h \| |u^j|^\alpha e_j\|_{\LL^2}^2 
 \leq \sum_{l=1}^7 T_{j,l},
\end{align}
where by the antisymmetry property \eqref{B} we have
\begin{align*}
T_{j,1}=&-\int_{t_{j-1}}^{t_j} \!\! \big\langle B\big(u(s)-u(t(j)), u(s)\big) \, , \, e_j\big\rangle ds, \quad
T_{j,2}=-\int_{t_{j-1}}^{t_j}\!\!  \big\langle B\big(e_j,u(s)\big) \, , \, e_j\big\rangle ds, \\
T_{j,3}=&-\int_{t_{j-1}}^{t_j} \!\! \big\langle B\big( u^j , u(s)-u^j\big) \, , e_j\big\rangle ds =
 -\int_{t_{j-1}}^{t_j}\!\!  \big\langle B\big( u^j, u(s)-u(t_j)\big) \, , \, e_j\big\rangle ds, \\
T_{j,4}=&-\nu \!\int_{t_{j-1}}^{t_j}\!\! \! \big( \nabla ( u(s)-u(t_j)) ,  \nabla e_j\big) ds, \\
T_{j,5}=&-a \int_{t_{j-1}}^{t_j} \!\! \big(|u(s)|^{2\alpha} u(s) - |u(t_j)|^{2\alpha} u(t_j)  ,  e_j\big) ds, \\
T_{j,6}=& \int_{t_{j-1}}^{t_j}\!\! \! \big([G(u(s))-G(u^{j-1}) dW(s), e_j-e_{j-1}\big), \\
T_{j,7}=& \int_{t_{j-1}}^{t_j}\!\!  \big([G(u(s))-G(u^{j-1}) \big] dW(s), e_{j-1}\big).
\end{align*}
We next prove upper estimates of  the terms $T_{j,l}$ for $l=1, ...,5$, and of the expected value of $T_{j,6}$ and $T_{j,7}$.

Using the H\"older inequality with exponents $2,3,6$, the Sobolev embedding $\HH^1\subset \LL^6$ and the Gagliardo Nirenberg inequality \eqref{GagNir},
 we deduce for $\epsilon_1>0$
 \begin{align}		\label{5.3}
 |T_{j,1}|\leq & \int_{t_{j-1}}^{t_j} \|u(s)-u(t_j)\|_{\LL^3} \|\nabla u(s)\|_{\LL^2} \|e_j\|_{\LL^6} ds \nonumber \\
 \leq & C_6 C_3  \| e_j\|_{\HH^1} \int_{t_{j-1}}^{t_j} \|u(s)-u(t_j)\|_{\LL^2}^{\frac{1}{2}} \| \nabla [u(s)-u(t_j)]\|_{\LL^2}^{\frac{1}{2}}  \|\nabla u(s)\|_{\LL^2}
  ds	\nonumber \\
 \leq & \epsilon_1 \nu h \|e_j\|_{\HH^1}^2 \nonumber \\
 & + \frac{(C_6 C_3)^2}{4\epsilon_1 \nu} \sup_{s\in [0,T]} \|u(s)\|_V^2 
 \Big( \int_{t_{j-1}}^{t_j}\!\! \|u(s)-u(t_j)\|_{\LL^2}^2ds\Big)^{\frac{1}{2}}  \Big( \int_{t_{j-1}}^{t_j}\!\! \|\nabla[u(s)-u(t_j)]\|_{\LL^2}^2ds\Big)^{\frac{1}{2}} 
 \nonumber \\
 \leq & \epsilon_1 \nu h \big[ \|e_j\|_{\LL^2}^2 + \| \nabla e_j\|_{\LL^2}^2\big]  + \frac{(C_6C_3)^4}{64 \epsilon_1^2 \nu^2} \sup_{s\in [0,T]} \|u(s)\|_V^4
 \int_{t_{j-1}}^{t_j}\!\! \|u(s)-u(t_j)\|_{\LL^2}^2ds 	\nonumber \\
 &+ \int_{t_{j-1}}^{t_j}\!\! \|\nabla[u(s)-u(t_j)]\|_{\LL^2}^2ds,
 \end{align}
 where the last inequalities are deduced from the Cauchy Schwarz and Young inequalities.
 
 Let $T_{j,2}=-T_{j,2,1} - T_{j,2,2}+T_{j,2,3}$, where 
 \begin{align*}  T_{j,2,1}=&\int_{t_{j-1}}^{t_j} \big\langle B\big( e_j,u(t_j)\big) , e_j\big\rangle ds, \quad  T_{j,2,2}=\int_{t_{j-1}}^{t_j} 
 \big\langle B\big( e_j,u(s)-u(t_j)\big) , u(t_j)\big\rangle ds, \\
 T_{j,2,3}=&\int_{t_{j-1}}^{t_j} \big\langle B\big( e_j,u(s)- u(t_j)\big) , u^j\big\rangle ds.
 \end{align*}
The antisymmetry \eqref{B} implies
\[ \big\langle B\big( e_j,u(t_j)\big) , e_j\big\rangle = - \big\langle B\big( e_j,e_j\big) , u(t_j)\big\rangle = -\sum_{k,l=1}^3 \int_{D}
(e_j)_k \partial_k(e_j)_l u(t_j)_l dx.\]
Hence the upper estimate \eqref{fgh2} with $f=u(t_j)_l$, $g=(e_j)_k$ and $h=\partial_k(e_j)_l$ yields for $\epsilon_2, \bar{\epsilon}_2>0$
\begin{align*} \big| \big\langle B\big( e_j,u(t_j)\big) , e_j\big\rangle \big| \leq &\epsilon_2 \nu \sum_{k,l} \|\partial_k(e_j)_l\|_{\LL^2}^2 +
\sum_{k,l} \frac{\bar{\epsilon}_2 a \kappa}{4\epsilon_2 \nu} \| |u(t_j)_l|^\alpha (e_j)_k\|_{\LL^2}^2 \\
& + \frac{C_\alpha}{\epsilon_2 \nu (\bar{\epsilon}_2 a \kappa)^{\frac{1}{\alpha -1}}} \|(e_j)_k\|_{\LL^2}^2,
\end{align*} 
which implies
\[ |T_{j,2,1}|\leq  \epsilon_2 \nu \, h \, \|\nabla e_j\|_{\LL^2}^2 + \frac{\bar{\epsilon}_2 a \kappa}{4\epsilon_2 \nu}\,   h\,  \| |u(t_j)|^\alpha e_j\|_{\LL^2}^2
+ \frac{C(\alpha, \nu, a, \kappa)}{\epsilon_2  (\bar{\epsilon}_2 )^{\frac{1}{\alpha -1}}} \, h\,  \|e_j\|_{\LL^2}^2.
\]
Using a  similar computation based on \eqref{fgh2} with  $f=u(t_j)_l$, $g=(e_j)_k$  and $h=\partial_k[u(s)-u(t_j)]_l$ for $k,l=1,2,3$, 
 summing on $k,l$  and integrating on the time interval $[t_{j-1}, t_j]$, we obtain for $\tilde{\epsilon}_2>0$
\[ |T_{j,2,2}|\leq   \int_{t_{j-1}}^{t_j} \!\! \| \nabla [ u(s)-u(t_j)]\|_{\LL^2}^2 ds  
+  \,  \frac{\tilde{\epsilon}_2 a \kappa}{4}\,   h\,  \| |u(t_j)|^\alpha e_j\|_{\LL^2}^2
+  \frac{\bar{C}(\alpha,  a, \kappa)}{ (\tilde{\epsilon}_2 )^{\frac{1}{\alpha -1}}} \, h \, \| e_j\|_{\LL^2}^2  .
\]
Replacing  $f=u(t_j)$ by $f=u^j$  in the above estimate, we obtain
\[ |T_{j,2,3}|\leq  \int_{t_{j-1}}^{t_j} \! \| \nabla [ u(s)-u(t_j)]\|_{\LL^2}^2 ds 
+ \frac{\tilde{\epsilon}_2 a \kappa}{4}\,   h\,  \| |u^j|^\alpha e_j\|_{\LL^2}^2
+ \frac{\bar{C}(\alpha,  a, \kappa)}{ (\tilde{\epsilon}_2 )^{\frac{1}{\alpha -1}}}  h \, \| e_j\|_{\LL^2}^2. 
\]
The three previous inequalities imply for $\epsilon_2, \bar{\epsilon}_2, \tilde{\epsilon}_2>0$, 
\begin{align}	\label{5.4}
|T_{j,2}|\leq & \;  \Big[  
\frac{C(\alpha, \nu, a, \kappa)}{\epsilon_2  (\bar{\epsilon}_2 )^{\frac{1}{\alpha -1}}}+  
  \frac{2\bar{C}(\alpha,  a, \kappa)}{ (\tilde{\epsilon}_2 )^{\frac{1}{\alpha -1}}} \Big]  \, h\,  \|e_j\|_{\LL^2}^2 
+ \epsilon_2 \nu \, h\, \|\nabla e_j\|_{\LL^2}^2 
+  \frac{\tilde{\epsilon_2}}{4} a \kappa \, h\,  \| |u^j|^\alpha e_j\|_{\LL^2}^2  \nonumber \\
& +\Big[ \frac{\bar{\epsilon}_2}{4\epsilon_2 \nu} + \frac{\tilde{\epsilon_2}}{4}\Big] a \kappa \, h\, 
 \| |u(t_j)|^\alpha e_j\|_{\LL^2}^2		  +   2  \int_{t_{j-1}}^{t_j}\!  \| \nabla [ u(s)-u(t_j)]\|_{\LL^2}^2 ds.
\end{align}
Using once more \eqref{fgh2} with $f=(u^j)_k$, $g=(e_j)_l$ and $h=\partial_k\big( [u(s)-u(t_j)]_l\big)$ for $k,l=1,2,3$, and summing on $k,l$, 
we obtain for $\epsilon_3>0$, 
\[ \big| \big\langle B\big( u^j, u(s)-u(t_j)\big) , e_j \big\rangle \big| \leq \|\nabla [u(s)-u(t_j)|\|_{\LL^2}^2 + 
\frac{\epsilon_3 a\kappa}{4}\,  \||u^j|^\alpha e_j\|_{\LL^2}^2  + \frac{C_\alpha}{(\epsilon_3 a \kappa)^{\frac{1}{\alpha -1}}}  \|e_j\|_{\LL^2}^2.
\]
 Integrating on $[t_{j-1}, t_j]$ we deduce for $\epsilon_3>0$ 
\begin{equation} 	\label{5.6}
 |T_{j,3}|\leq  \frac{C_\alpha}{(\epsilon_3 a \kappa)^{\frac{1}{\alpha -1}}} \, h\, \|e_j\|_{\LL^2}^2 + 
 \frac{\epsilon_3 a\kappa}{4}\, h\, \||u^j|^\alpha e_j\|_{\LL^2}^2  + \int_{t_{j-1}}^{t_j} \| \nabla [u(s)-u(t_j)|\|_{\LL^2}^2 ds . 
\end{equation}
The Cauchy-Schwarz and Young inequalities imply that for $\epsilon_4>0$,
\begin{equation}		\label{5.7}
|T_{j,4}|\leq \epsilon_4 \nu \, h\, \|\nabla e_j\|_{\LL^2}^2 + \frac{\nu}{4 \epsilon_4} \int_{t_{j-1}}^{t_j} \| \nabla [ u(s)-u(t_j)]\|_{\LL^2}^2 ds.
\end{equation}
Since $\big| |f|^{2\alpha} f - |g|^{2\alpha} g\big| \leq C(\alpha) |f-g| \big( |f|^{2\alpha} + |g|^{2\alpha} \big)$, the H\"older inequality with exponents
2,3 and 6 implies 
\begin{align*}	
\big| \big( |u(s)|^{2\alpha} u(s) - |u(t_j)|^{2\alpha} u(t_j) , e_j\big) \big| \leq & \, C(\alpha)\int_{\RR^3} \big[ |u(s)|^{2\alpha} + |u(t_j)|^{2\alpha} \big]
|u(s)-u(t_j)| |e_j| dx \\
\leq & \, C(\alpha) \big[ \|u(s)\|_{\LL^{4\alpha}}^{2\alpha} +      \|u(t_j)\|_{\LL^{4\alpha}}^{2\alpha}  \big] \|u(s)-u(t_j)\|_{\LL^3} \|e_j\|_{\LL^6} .
\end{align*} 
The Sobolev embedding $\HH^1\subset \LL^6$ and the Gagliardo Nirenberg inequality  \eqref{GagNir} yield for $\epsilon_5>0$ 
\begin{align}		\label{5.8}
|T_{j,5}|\leq & C(\alpha)  \sup_{s\in [0,T]} \|u(s)\|_V^{2\alpha} \int_{t_{j-1}}^{t_j} \|e_j\|_{\HH^1} \|u(s)-u(t_j)\|_{\LL^2}^{\frac{1}{2}}
\|\nabla [ u(s)-u(t_j)]\|_{\LL^2}^{\frac{1}{2}} ds 	\nonumber \\
\leq & \epsilon_5 \nu \, h \, \big[ \|e_j\|_{\LL^2}^2 + \| \nabla e_j\|_{\LL^2}^2 \big] +
\frac{C(\alpha)^2}{8 \epsilon_5 \nu } \sup_{s\in [0,T]} \|u(s)\|_V^{8\alpha}  \int_{t_{j-1}}^{t_j} \|u(s)-u(t_j)\|_{\LL^2}^2 ds  	\nonumber \\
& +
\frac{C(\alpha)^2}{8 \epsilon_5 \nu }  \int_{t_{j-1}}^{t_j} \| \nabla [u(s)-u(t_j)]\|_{\LL^2}^2 ds, 
\end{align}
where the last upper estimate is deduced from the H\"older inequality with exponents 2,4 and 4 and the Young inequality. 

Fix $J\in \{ 1, 2, ..., N\}$; adding the inequalities \eqref{5.2} for $j=1, ..., J$, using the identity $e_0=0$ and the upper estimates \eqref{5.3}--\eqref{5.8}
we deduce that for any positive numbers $\epsilon_j, j=1, ...,5$,  $\bar{\epsilon}_2$ and $\tilde{\epsilon}_2$, we have
\begin{align}	\label{5.8Bis}
\frac{1}{2}& \|e_J\|_{\LL^2}^2 + \frac{1}{2} \sum_{j=1}^J \|e_j-e_{j-1}\|_{\LL^2}^2 + \nu\, h\, \sum_{j=1}^J \|\nabla e_j\|_{\LL^2}^2 +
a \kappa \, h \, \sum_{j=1}^J \Big[ \| |u(t_j)|^\alpha e_j\|_{\LL^2}^2 + \| |u^j|^{\alpha} e_j\|_{\LL^2}^2\Big] \nonumber  \\
&\leq \sum_{j=1}^J \sum_{l=6}^7 T_{j,l} +  \Big[ \epsilon_1\nu +  \frac{2 \, \bar{C}(\alpha,a,\kappa)}{(\tilde{\epsilon}_2)^{\frac{1}{\alpha-1}}} 
+\frac{C(\alpha,\nu,a,\kappa)}{\epsilon_2 \nu (\bar{\epsilon}_2)^{\frac{1}{\alpha -1}}} 
+ \frac{C_\alpha}{(\epsilon_3 a \kappa)^{\frac{1}{\alpha -1}}} +  \epsilon_5 \nu\Big] \, h\, \sum_{j=1}^J \|e_j\|_{\LL^2}^2  \nonumber \\
& \quad + \Big( \epsilon_1+\epsilon_2 + \epsilon_4+ \epsilon_5\Big) \nu \, h\, \sum_{j=1}^J \|\nabla e_j\|_{\LL^2}^2
+ \Big( \frac{\bar{\epsilon}_2}{4\epsilon_2 \nu} + \frac{\tilde{\epsilon}_2}{4}\Big) a\kappa \, h\, \sum_{j=1}^J \| |u(t_j)|^\alpha e_j\|_{\LL^2}^2 \nonumber \\
&\quad +    \frac{\tilde{\epsilon}_2+ \epsilon_3}{4}   a\kappa \, h\, \sum_{j=1}^J \| |u^j|^\alpha e_j\|_{\LL^2}^2
+ \frac{(C_6C_3)^4}{64 \epsilon_1^2 \nu^2} \sup_{s\in [0,T]} \|u(s)\|_V^4 \sum_{j=1}^J \int_{t_{j-1}}^{t_j} \|u(s)-u(t_j)\|_{\LL^2}^2 ds \nonumber  \\
&\quad + \frac{C(\alpha)^2}{8\epsilon_5 \nu}\sup_{s\in [0,T]} \|u(s)\|_V^{8\alpha}  \sum_{j=1}^J \int_{t_{j-1}}^{t_j} \|u(s)-u(t_j)\|_{\LL^2}^2 ds \nonumber \\
&\quad +  \Big[ 4+  \frac{\nu}{4\epsilon_4} + \frac{C(\alpha)^2}{8\epsilon_5 \nu}    
 \Big] \sum_{j=1}^J \int_{t_{j-1}}^{t_j}
 \| \nabla [u(s)-u(t_j)|\|_{\LL^2}^2 ds.  
\end{align} 
Choose positive $\epsilon_1, \epsilon_2, \epsilon_4$ and $\epsilon_5$ such that $\epsilon_1+\epsilon_2+\epsilon_4+\epsilon_5 \leq \frac{1}{2}$; then
choose positive  $\bar{\epsilon}_2$, $\tilde{\epsilon}_2$ and $\epsilon_3$ such that $\frac{\bar{\epsilon}_2}{4\epsilon_2\nu} + \frac{\tilde{\epsilon}_2}{4}
\leq 1$ and $\frac{\tilde{\epsilon}_2 +\epsilon_3 }{4}\leq 1$.  We deduce the existence of positive constants  $C_i$, $i=1,2,3$ 
depending on $\nu, a, \kappa$, $ \epsilon_j$ for $j=1,...,5$, $\bar{\epsilon}_2$ and $\tilde{\epsilon}_2$, such that
\begin{align*}	
\frac{1}{2}& \|e_J\|_{\LL^2}^2 + \frac{1}{2} \sum_{j=1}^J \|e_j-e_{j-1}\|_{\LL^2}^2 + \frac{\nu}{2} \, h\, \sum_{j=1}^J \|\nabla e_j\|_{\LL^2}^2 
\leq  C_1 \, h\, \sum_{j=1}^J \|e_j\|_{\LL^2}^2 \\
& + C_2 \Big[ 1+  \sup_{s\in [0,T]} \|u(s)\|_V^{8\alpha} \Big] \sum_{j=1}^J \int_{t_{j-1}}^{t_j}  \|u(s)-u(t_j)\|_{\LL^2}^2 ds  \\
& + C_3 
 \sum_{j=1}^J \int_{t_{j-1}}^{t_j} \| \nabla [u(s)-u(t_j)\|_{\LL^2}^2 ds
 + \sum_{j=1}^J \sum_{l=6}^7 T_{j,l} . 
\end{align*}
Let $N$ be large enough to ensure $C_1 \frac{T}{N} < \frac{1}{4}$. 
Note that for non negative numbers $\{ x(J), y(J);$ 
  \linebreak[4]  $J=1, ..., N\}$ we have $\frac{1}{2} \big[ \sup_{J\leq N} a(J )+ \sup_{J\leq N} b(J )\big]
 \leq \sup_{J\leq N} [a(J)+b(J)]$.
Therefore, using this upper estimate and then 
taking expected values in the above inequality, using the Cauchy-Schwarz
and   H\"older inequalities with  conjugate exponents $p,q\in (1,\infty)$, we deduce 
 \begin{align}		\label{5.10}
 \frac{1}{8}& \EE\Big(\max_{J\leq N} \|e_J\|_{\LL^2}^2 \Big)+ \frac{1}{4} \sum_{j=1}^N \EE(\|e_j-e_{j-1}\|_{\LL^2}^2) 
 + \frac{\nu}{4} \, h \sum_{j=1}^N \EE( \|\nabla e_j\|_{\LL^2}^2 )
 \leq C_1\, h  \sum_{j=0}^{N-1} \EE( \|e_j\|_{\LL^2}^2)  	\nonumber \\
&\quad + C_2 \Big\{ 1+ \EE\Big( \sup_{s\in [0,T]}  \|u(s)\|_V^{16\alpha}\Big) \Big\}^{\frac{1}{2}} 
\Big\{ N\, h \, \sum_{j=1}^N \EE \int_{t_{j-1}}^{t_j} \|u(s)-u(t_j)\|_{\LL^2}^4 ds \Big\}^{\frac{1}{2}} 	\nonumber \\
&\quad + C_3 \EE\Big( \sum_{j=1}^N \int_{t_{j-1}}^{t_j} \|\nabla [u(s)-u(t_j)\|_{\LL^2}^2 ds \Big)
  +\EE\Big( \sum_{k=1}^N |T_{j,6}| \Big) + \EE\Big(  \max_{K\leq N} \sum_{j=1}^K  T_{j,7}\Big) .
 \end{align}
We next find upper estimates of the expected value of the sum of the stochastic terms $T_{j,l},  l=6,7$. 

For $j\in \{1, ..., N\}$, the Cauchy-Schwarz and Young inequalities,  the Lipschitz condition \eqref{LipG}, the Cauchy-Schwarz and Young inequalities
 imply for $\epsilon_6>0$ 
\begin{align}		\label{5.11}
\EE\big|&T_{j,6}\big| \leq  \EE\Big( \Big\| \int_{t_{j-1}}^{t_j}\!\! \big[ G(u(s))-G(u^{j-1})\big] dW(s)\Big\|_{\LL^2} \, \|e_j-e_{j-1}\|_{\LL^2}\Big) 
\nonumber \\
&\leq  \epsilon_6  \, \EE\big( \|e_j-e_{j-1}\|_{\LL^2}^2\big) + \frac{2}{4\epsilon_6} \EE\int_{t_{j-1}}^{t_j}\!\! \big[ L \|u(s)-u(t_{j-1})\|_{\LL^2}^2
+ L \|e_{j-1}\|_{\LL^2}^2 \big]\,  \mbox{\rm Tr} Q\, ds
\nonumber \\
&\leq  \epsilon_6 \, \EE\big( \|e_j-e_{j-1}\|_{\LL^2}^2\big) + h\,  \frac{L\, \mbox{\rm Tr } Q}{2\epsilon_6} \EE(\|e_{j-1}\|_{\LL^2}^2) +
\frac{L\, \mbox{\rm Tr} Q}{2\epsilon_6} \EE\int_{t_{j-1}}^{t_j}\!\!  \|u(s)-u(t_{j-1})\|_{\LL^2}^2\,  ds.
\end{align}
Using the Davis inequality and the Lipschitz condition \eqref{LipG}, we deduce that for  $\epsilon_7>0$
\begin{align}	\label{5.12}
\EE\Big( \max_{K\leq N} &\sum_{j=1}^N T_{j,7}\Big)  \leq 3 \EE\Big( \Big\{ \sum_{j=1}^J \int_{t_{j-1}}^{t_j} 
 \| G\big(u(s)\big)- G\big(u^{j-1}  \big)\|_{\mathcal L}^2 
\; \|e_{j-1}\|_{\LL^2}^2 \, \mbox{\rm Tr} Q\, ds \Big\}^{\frac{1}{2}} \Big)		\nonumber \\
\leq & \; 3 \; \EE\Big( \max_{0\leq j\leq N-1} \|e_j\|_{\LL^2} \Big\{  \sum_{j=1}^N \int_{t_{j-1}}^{t_j} \| G\big(u(s)\big)-
 G\big(u^{j-1})  \big)\|_{\mathcal L}^2 \, \mbox{\rm Tr} Q
\;  ds \Big\}^{\frac{1}{2}} \Big)		\nonumber \\
\leq &\;  \epsilon_7 \; \EE\Big( \max_{1\leq j\leq N} \|e_j\|_{\LL^2}^2\Big) 
+ \frac{18\, L\, \mbox{\rm Tr} Q}{4\epsilon_7} \EE\Big( \sum_{j=1}^N \int_{t_{j-1}}^{t_j}\!\!  \big[  \|u(s)-u(t_{j-1})\|_{\LL^2}^2 
+ \|e_{j-1}\|_{\LL^2}^2\big] \, ds\Big),
\end{align}
where in the last inequality we have used  $e_0=0$  and Young's inequality. 

Choose  $\epsilon_6  = \frac{1}{4}$ and $\epsilon_7 = \frac{1}{16}$; the upper estimates \eqref{5.10} -- \eqref{5.12} imply
\begin{align*}
 \frac{1}{16} \EE\Big( & \max_{J\leq N} \|e_J\|_{\LL^2}^2 \Big)  
 + \frac{\nu}{4} \, h\, \sum_{j=1}^N \EE( \|\nabla e_j\|_{\LL^2}^2 )  
\leq  \big( C_1+74\, L\, {\rm Tr}Q) \, h\, \sum_{j=0}^{N-1}  \EE(\|e_j\|_{\LL^2}^2) \\
&\quad +   C_2 T  \Big\{1+  \EE\Big(   \sup_{s\in [0,T]}  \|u(s)\|_V^{16\alpha}\Big) \Big\}^{\frac{1}{2}} 
\Big\{  \sum_{j=1}^N \EE \int_{t_{j-1}}^{t_j} \|u(s)-u(t_j)\|_{\LL^2}^4 ds \Big\}^{\frac{1}{2}}\\
&\quad + C(T,L,\mbox{\rm Tr } Q)  
\sum_{j=1}^N \int_{t_{j-1}}^{t_j} \!\! \EE\big(\|u(s)-u(t_{j-1})\|_{\LL^2}^2 \big)ds\\
&\quad + C_3 \EE\Big( \sum_{j=1}^N \int_{t_{j-1}}^{t_j}\!\! \|\nabla [u(s)-u(t_j)]\|_{\LL^2}^2 ds \Big)  .
\end{align*} 
Let $\lambda \in (0,1)$ and set $\delta = \frac{1}{4} (1-\lambda)$. 
The moment estimates \eqref{3.9} and \eqref{3.10} imply 
\begin{align}	\label{5.13}
 \frac{1}{16} \EE\Big(&\max_{j\leq N} \|e_j\|_{\LL^2}^2 \Big) 
 + \frac{\nu}{4} \, h\, \sum_{j=1}^N \EE( \|\nabla e_j\|_{\LL^2}^2 ) 
\leq  \big( C_1+74\, L\, {\rm Tr} Q) \, h\, \sum_{j=0}^{N-1}  \EE(\|e_j\|_{\LL^2}^2) 	\nonumber \\
&\quad + C(T) \big\{ 1+\EE(\|u_0\|_V^{16\alpha})\big\}^{\frac{1}{2}} \, h^\lambda + C \Big[ 1+\EE\Big(\|u_0\|_V^{\frac{16\alpha +2+8\delta }{1+4\delta}}\Big)\Big]\, h^\lambda	
\end{align}
for some constant $C:=C(T, \nu, \alpha, a,p,{\rm Tr}Q)$. 
Note that for $\delta\in \big(0, \frac{1}{32\alpha -4}\big)$ we have $\frac{16\alpha+2+8\delta}{1+4\delta} \geq 16\alpha$. 
 Neglecting  the second  term in the left hand side of \eqref{5.13} and using the discrete Gronwall lemma, we deduce that, for some positive
  constants $C$ (resp. $C_1$) depending on $T,\nu, \alpha, a, \mbox{\rm Tr }Q$  and
  $ \EE\Big(\|u_0\|_V^{\frac{16\alpha +2+8\delta }{1+4\delta}}\Big) $  (resp.  depending on $\nu, \alpha, a, \kappa$)  such that 
\[ \EE\Big( \max_{j\leq N} \|e_j\|_{\LL^2}^2 \Big) \leq C \, h^\lambda\, e^{16(C_1+74\, L\, {\rm Tr}Q)T}.\]
Plugging this inequality in \eqref{5.13} we deduce\eqref{strong-full->1}; this completes the proof when $\alpha\in (1,\frac{3}{2}]$.
\medskip

 (ii) We next let $\alpha=1$ and assume $4\nu a >1$ and $4\nu a \kappa >1$; we  only point out the differences in the proof. 

We have to use a different argument to obtain upper estimates of the terms $\{ T_{j,2,i},i=1,2,3\}$ and $T_{j,3}$. 
The Cauchy-Schwarz and Young inequalities prove that for $\epsilon_2, \bar{\epsilon}_2, \tilde{\epsilon}_2>0$, 
\begin{align*} 	
 |T_{j,2,1}| &\leq \epsilon_2 \nu \, h\, \|\nabla e_j\|_{\LL^2}^2 + \frac{1}{4\epsilon_2 \nu} \, h\, \| |u(t_j)| e_j\|_{\LL^2}^2,\\
 |T_{j,2,2}|&\leq \bar{\epsilon}_2 \, h\, \||u(t_j)| e_j\|_{\LL^2}^2 + \frac{1}{4\bar{\epsilon}_2} \int_{t_{j-1}}^{t_j} \| \nabla [ u(s)-u(t_j)] \|_{\LL^2}^2 ds , \\
 |T_{j,2,3}|&\leq \tilde{\epsilon}_2 \, h\, \||u^j | e_j\|_{\LL^2}^2 + \frac{1}{4\tilde{\epsilon}_2} \int_{t_{j-1}}^{t_j} \| \nabla [ u(s)-u(t_j)] \|_{\LL^2}^2 ds .
 \end{align*}  
This implies 
\begin{align}		\label{Tj2Bis}
|T_{j,2}| \leq &\;  \epsilon_2 \nu \, h\, \|\nabla e_j\|_{\LL^2}^2 +\Big(  \frac{1}{4\epsilon_2 \nu}+ \bar{\epsilon_2}\Big)  \, h\, \| |u(t_j)| e_j\|_{\LL^2}^2
+ \tilde{\epsilon}_2 \, h\, \||u^j | e_j\|_{\LL^2}^2  \nonumber \\
&+ \Big( \frac{1}{4\bar{\epsilon}_2}+ \frac{1}{4\tilde{\epsilon}_2}\Big)  \int_{t_{j-1}}^{t_j} \| \nabla [ u(s)-u(t_j)] \|_{\LL^2}^2 ds. 
\end{align} 
Using once more the Cauchy-Schwarz and Young inequalities, we obtain for $\epsilon_3>0$ 
\begin{equation}		\label{Tj3Bis}
|T_{j,3}|\leq \epsilon_3\, h\, \||u^j|e_j\|_{\LL^2}^2 + \frac{1}{4\epsilon_3}   \int_{t_{j-1}}^{t_j} \| \nabla [ u(s)-u(t_j)] \|_{\LL^2}^2 ds.
\end{equation}
The upper estimates  \eqref{5.3},  \eqref{Tj2Bis}, \eqref{Tj3Bis}, \eqref{5.7} and   \eqref{5.8} 
imply 
 for any positive numbers $\epsilon_j, j=1, ...,5$, $\bar{\epsilon}_2$ and $\tilde{\epsilon}_2$
\begin{align}	\label{**}
\frac{1}{2}& \|e_J\|_{\LL^2}^2 + \frac{1}{2} \sum_{j=1}^J \|e_j-e_{j-1}\|_{\LL^2}^2 + \nu\, h\, \sum_{j=1}^J \|\nabla e_j\|_{\LL^2}^2 +
a \kappa \, h \, \sum_{j=1}^J \Big[ \| |u(t_j)| e_j\|_{\LL^2}^2 + \| |u^j| e_j\|_{\LL^2}^2\Big] \nonumber  \\
&\leq \sum_{j=1}^J \sum_{l=6}^7 T_{j,l} +  \Big[ \epsilon_1\nu +  \epsilon_5 \nu\Big] \, h\, \sum_{j=1}^J \|e_j\|_{\LL^2}^2  
+ \Big( \epsilon_1+\epsilon_2 + \epsilon_4+ \epsilon_5\Big) \nu \, h\, \sum_{j=1}^J \|\nabla e_j\|_{\LL^2}^2 \nonumber \\
& \quad 
+ \Big( \frac{1}{4\epsilon_2 \nu} + \bar{\epsilon}_2\Big)  \, h\, \sum_{j=1}^J \| |u(t_j)| e_j\|_{\LL^2}^2 
+  \big( \tilde{\epsilon_2} + \epsilon_3\big)  \, h\, \sum_{j=1}^J \| |u^j| e_j\|_{\LL^2}^2
\nonumber \\
&\quad 
+  \frac{(C_6C_3)^4}{64 \epsilon_1^2 \nu^2}  \sup_{s\in [0,T]} \|u(s)\|_V^8  
\sum_{j=1}^J \int_{t_{j-1}}^{t_j} \|u(s)-u(t_j)\|_{\LL^2}^2 ds \nonumber  \\
&\quad + \Big[ 1   
+ \frac{1}{4\bar{\epsilon}_2} + \frac{1}{4\tilde{\epsilon}_2} +\frac{1}{4\epsilon_3} 
+ \frac{\nu}{4\epsilon_4}+ \frac{C(\alpha)^2}{8\epsilon_5 \nu}  \Big] \sum_{j=1}^J \int_{t_{j-1}}^{t_j}
 \| \nabla [u(s)-u(t_j)] \|_{\LL^2}^2 ds.  
\end{align} 
Fix $\epsilon \in (0,\frac{1}{2})$ such that $(1-2\epsilon)^2 4\nu a \kappa >1$,  let $\epsilon_2= 1-2\epsilon$,
 and then choose positive numbers $\epsilon_1, \epsilon_4$ and $\epsilon_5$ 
such that $\epsilon_1+\epsilon_2+\epsilon_4+\epsilon_5 = 1-\epsilon$.  Choose $\bar{\epsilon}_2\in (0, \epsilon a\kappa)$,
$\tilde{\epsilon}_2 + \epsilon_3 \leq  a\kappa$. 
The choice of $\epsilon_2$ and $\bar{\epsilon}_2$ implies   $\frac{1}{4\epsilon_2 \nu}+\bar{\epsilon}_2 < a\kappa$.  Therefore, 
\begin{align*}	
\frac{1}{2}& \|e_J\|_{\LL^2}^2 + \frac{1}{2} \sum_{j=1}^J \|e_j-e_{j-1}\|_{\LL^2}^2 + \epsilon \nu  \, h \sum_{j=1}^J \|\nabla e_j\|_{\LL^2}^2 
\leq  C_1 \, h\, \sum_{j=1}^J \|e_j\|_{\LL^2}^2 + \sum_{j=1}^J \sum_{l=6}^7 T_{j,l} \\
& \quad + C_2  \sup_{s\in [0,T]} \|u(s)\|_V^8  \sum_{j=1}^J \int_{t_{j-1}}^{t_j}  \|u(s)-u(t_j)\|_{\LL^2}^2 ds 
\sum_{j=1}^J\int_{t_{j-1}}^{t_j} \!  \| \nabla [u(s)-u(t_j)\|_{\LL^2}^2 ds 
\end{align*}
As in the case $\alpha \in (1,\frac{3}{2}]$, using \eqref{5.11} and \eqref{5.12} with  $\epsilon_6  = \frac{1}{4}$ and $\epsilon_7 = \frac{1}{16}$, 
we deduce 
\begin{align*}
& \frac{1}{16} \EE\Big(\sup_{J\leq N} \|e_J\|_{\LL^2}^2 \Big) 
 + \frac{\epsilon \nu}{4} \, h\, \sum_{j=1}^N \EE( \|\nabla e_j\|_{\LL^2}^2 ) 
\leq  \big( C_1+74\, L\, {\rm Tr}Q)  \, h\, \sum_{j=0}^{N-1}  \EE(\|e_j\|_{\LL^2}^2) \\
&\quad +   C_2 T  \Big\{1+  \EE\Big( \sup_{t\in [0,T]} \|u(s)\|_V^{16} \Big) \Big\}^{\frac{1}{2}} 
\Big\{  \sum_{j=1}^N \EE \int_{t_{j-1}}^{t_j} \|u(s)-u(t_j)\|_{\LL^2}^4 ds \Big\}^{\frac{1}{2}}\\
&\quad + C_3 \EE\Big( \sum_{j=1}^N \int_{t_{j-1}}^{t_j}\!\! \|\nabla [u(s)-u(t_j)]\|_{\LL^2}^2 ds \Big) . 
\end{align*} 
We conclude the proof as in the case $\alpha \in (1,\frac{3}{2}]$.  
\hfill $\Box$

\section{Appendix}		\label{Appendix} 
In this section, we provide the proof of the well-posedness result stated in section 
\ref{s-gwp}. 
\subsection{Proofs of preliminary estimates}	\label{Ap1}
The following results gather some estimates of the bilinear term, and more generally of the  non linear part in \eqref{3D-NS}. 
They are deduced from the Brinkman Forchheimer smoothing term. 
The proofs are somewhat similar to the corresponding ones in \cite{BeMi_anisotropic} in a different functional setting.

The next lemma gathers further properties of $B$. 
\begin{lemma}		\label{upperB} 
Suppose that  $\alpha \in [1,+\infty)$. 

(i) Let $u\in L^\infty(0,T;H)\cap L^{2\alpha+2}([0,T]\times D;\RR^3)$, $v\in X_0$. Then
\begin{align}
&\int_0^T \! \big| \langle B(u(t), u(t)) , v(t)\rangle \big| dt \leq  \| \nabla v\|_{L^2(0,T;\LL^2)} \; 
\underset{{t\in [0,T]}}{\mbox{\rm ess sup }} \|u(t)\|_{\LL^2}^{\frac{\alpha-1}{\alpha}}\;
\|u\|_{L^{2\alpha +2}([0,T]\times D;\RR^3)}^{\frac{\alpha +1}{\alpha}} \; T^{\frac{\alpha -1}{2\alpha}}. \label{1.11}\\
&\int_0^T\!  \big| \langle B(u(t), u(t))  - B(v(t),v(t)) , u(t)- v(t)\rangle \big| dt \leq 
\| \nabla v\|_{L^2(0,T;H)} \nonumber \\
&\qquad \qquad \qquad \qquad \times \; \underset{{t\in [0,T]}}{\mbox{\rm ess sup }} \|(u-v)(t)\|_{H}^{\frac{\alpha-1}{\alpha}}\;
\|u-v\|_{L^{2\alpha +2}([0,T]\times D;\RR^3)}^{\frac{\alpha +1}{\alpha}} \; T^{\frac{\alpha -1}{2\alpha}}. \label{1.12}
\end{align}
(ii) Let $u\in L^4(\Omega ; L^\infty(0,T;H))\cap L^{2\alpha +2}(\Omega_T\times D;\RR^3)$ and $v\in {\mathcal X}_0$. Then
\begin{align}
&\EE \int_0^T \! \big| \langle B(u(t), u(t)) , v(t)\rangle \big| dt \leq \Big\{ \EE \Big| \int_0^T \|\nabla v(t)\|_{\LL^2}^2 dt  \Big|^2  \Big\}^{\frac{1}{4}}  \; 
\Big\{ \EE\Big( \underset{{t\in [0,T]}}{\mbox{\rm ess sup }} \|u(t)\|_{H}^{4}\Big) \Big\}^{ \frac{\alpha-1}{4\alpha}} \nonumber \\
& \qquad \qquad \qquad \qquad \times 
\Big\{ \EE \int_0^T \!\! dt \int_{D} \! |u(t,x)|^{2\alpha +2} dx \Big\}^{\frac{1}{2\alpha}} \;  T^{\frac{\alpha -1}{2\alpha}}, 
\label{1.13} \\
&\EE \int_0^T\!  \big| \langle B(u(t), u(t))  - B(v(t),v(t)) , u(t)- v(t)\rangle \big| dt \leq T^{\frac{\alpha -1}{2\alpha}} \, 
 \Big\{ \EE \Big| \int_0^T \! \|\nabla v(t)\|_{\LL^2}^2 dt  \Big|^2  \Big\}^{\frac{1}{4}} , 
 \nonumber \\
& \qquad \times \Big\{ \EE\Big( \underset{{t\in [0,T]}}{\mbox{\rm ess sup }} \|(u-v)(t)\|_{H}^{4}\Big) \Big\}^{\frac{\alpha -1}{4 \alpha}}   
\; \Big\{ \EE \! \int_0^T \!\! dt\! \int_{D} \!  |(u-v)(t,x)|^{2\alpha +2} dx \Big\}^{\frac{1}{2\alpha}} .  \label{1.14}
\end{align}
\end{lemma}
\begin{proof} 
(i) Suppose $\alpha >1$. Using \eqref{fgh1} with $h=\partial_i v_j$, $f=u_i$ and $g=u_j$, we deduce
\begin{align*}
|\langle B(u,u),v\rangle | &=  |-\langle B(u,v),u\rangle | \leq \sum_{i,j=1}^3 \int_{D} |u_i(x) \partial_i v_j(x) u_j(x)| dx \\
&\leq \big\| |u| \, |u|^{\frac{1}{\alpha}}\big\|_{\LL^{2\alpha}} \; \big\| |u|^{1-\frac{1}{\alpha}} \big\|_{\LL^{\frac{2\alpha}{\alpha-1}}}\; \| \nabla v\|_{\LL^2}.
\end{align*}
Integrating on the time interval $[0,T]$ and using the Cauchy-Schwarz inequality, we obtain
\[ \int_0^T \! \big| \langle B(u(t), u(t)) , v(t)\rangle \big| dt \leq \underset{{t\in [0,T]}}{\mbox{\rm ess sup }} \|u(t)\|_{H}^{\frac{\alpha-1}{\alpha}}\; 
\Big( \int_0^T\!\!  \|u(t)\|_{\LL^{2\alpha +2}}^{\frac{2\alpha +2}{\alpha}} dt \Big)^{\frac{1}{2}} 
\Big( \int_0^T \!\! \|\nabla v(t)\|_{\LL^2}^2 dt \Big)^{\frac{1}{2}}.
\]
H\"older's inequality implies 
\[ \int_0^T\!\!  \|u(t)\|_{\LL^{2\alpha +2}}^{\frac{2\alpha +2}{\alpha}} dt \leq 
\|u\|_{L^{2\alpha+2}([0,T]\times D;\RR^3)}^{\frac{2\alpha+2}{\alpha}}\; T^{\frac{\alpha-1}{\alpha}}.\]
 This completes the proof of \eqref{1.11} for $\alpha >1$.

 If $\alpha=1$, since $ |\langle B(u,u),v\rangle |  \leq \big\|  u\big\|_{\LL^4}^2 \; \| \nabla v\|_{\LL^2}$, 
 a straightforward computation implies \eqref{1.11}.

Since $\langle B(u,u)-B(v,v) \, , \, u-v\rangle = \langle B(u-v,v)\, , \, u-v\rangle$, using the antisymmetry \eqref{B} it is easy to see that 
 the upper estimate \eqref{1.11} implies \eqref{1.12}.
 \smallskip

(ii) For $\alpha >1>\frac{2}{3}$, we have $\frac{4\alpha}{3\alpha -2}>1$. 
Using  H\"older's inequality  for the expected value with exponents  4, 
$\frac{4\alpha}{3\alpha -2}$ and $2\alpha$ in \eqref{1.11}, we deduce 
\begin{align*}
 \EE \int_0^T \! \big| \langle B(u(t), u(t)) , v(t)\rangle \big| dt \leq &
\Big\{ \EE  \Big( \|\nabla v\|_{L^2(0,T;\LL^2)}^4 \Big)\Big\} ^{\frac{1}{4 }}
\Big\{  \EE\Big(  \underset{{t\in [0,T]}}{\mbox{\rm ess sup }} \|u(t)\|_{\LL^2}^{\frac{4(\alpha-1)}{3\alpha-2}} \Big)
 \Big\}^{\frac{3\alpha-2}{4\alpha} } \\ 
 &\times \Big\{ E  \int_0^T dt \int_{D}  |u(t,x)|^{2\alpha +2} \, dx \Big\}^{\frac{1}{2\alpha}}\; T^{\frac{\alpha-1}{2\alpha}} .
\end{align*}
Since $\alpha >\frac{1}{2}$ we have $\frac{4(\alpha -1)}{3\alpha -2} < 4$; 
this completes the proof of \eqref{1.13} for $\alpha >1$.

 For $\alpha=1$, using the antisymmetry \eqref{B}, and twice the Cauchy-Schwarz inequality, we deduce 
\begin{align*}
\EE \int_0^T \! \big| \langle B(u(t), u(t)) &, v(t)\rangle \big| dt \leq 
 \Big\{ \EE \int_0^T  \|\nabla v(t)\|_{\LL^2}^2dt   \Big\}^{\frac{1}{2}} 
\Big\{ \EE \int_0^T  \|u(t)\|_{\LL^4}^4dt  \Big\}^{\frac{1}{2}} \\
&\leq \Big\{ \EE \Big| \int_0^T  \|\nabla v(t)\|_{\LL^2}^2 dt \Big|^2  \Big\}^{\frac{1}{4}} 
\Big\{ \EE \int_0^T  \|u(t)\|_{\LL^4}^4dt  \Big\}^{\frac{1}{2}}.
\end{align*}
This completes the proof of \eqref{1.13}.

A similar argument based on the identity $\langle B(u,u)-B(v,v) \, , \, u-v\rangle = \langle B(u-v,v)\, , \, u-v\rangle$ shows   \eqref{1.14}. 
\end{proof} 
\medskip

We next prove upper estimates for the gradient of the bilinear term. 
\begin{lemma} 		\label{upper_nablaB}
(i) There exists a positive constant $C$ such that for $\alpha \in  (1,\infty)$, some  constant $C_\alpha >0$, any constants 
$\varepsilon_0, \varepsilon_1 >0$ we have for $u\in X_1$,
\begin{equation} 		\label{1.18}
|\langle  A^{1/2}  B(u,u)\, , \, A^{1/2}  u\rangle | \leq C \Big[ \varepsilon_0 \|Au\|_{\LL^2}^2 +
 \frac{\varepsilon_1}{4\varepsilon_0} \big\| |u|^\alpha \nabla u \big\|_{\LL^2}^2 + \frac{C_\alpha}{\varepsilon_0 \varepsilon_1^{\frac{1}{\alpha-1}}}
 \|\nabla u\|_{\LL^2}^2 \Big].
\end{equation}
(ii) Let $\alpha =1$; for every $\epsilon >0$,  we have for some constant $C>0$ and any $u\in X_1$
\begin{equation} 		\label{1.18Bis}
|\langle  A^{1/2} B(u,u)\, , \,A^{1/2} u \rangle | \leq   \varepsilon \|Au\|_{\LL^2}^2 +
 \frac{1}{4\varepsilon} \big\| |u| \nabla u \big\|_{\LL^2}^2 .
 \end{equation}
\end{lemma} 
 \begin{proof}
 (i) Let $\alpha >1$ and  $u\in X_1$. Then 
\[ \langle A^{1/2} B(u,u)\, , \,A^{1/2} u \rangle = \sum_{i,j,k=1}^3 \int_{D} \partial_k\big[ u_i \, \partial_i u_j \big]\,  \partial_k u_j dx = T_1+T_2,\]
where, using the antisymmetry property \eqref{B},  we get 
\begin{align*}
T_1= & \sum_{i,j,k=1}^3 \int_{D} \partial_k u_i  \; \partial_i u_j \;  \partial_k u_j dx, \\
T_2= & \sum_{i,j,k=1}^3 \int_{D} u_i \; \partial_k \partial_i u_j \; \partial_k u_j dx  = \sum_{k=1}^3 \langle B(u,\partial_k u)\, , \, \partial_k u\rangle = 0. 
\end{align*}
Using integration by parts, we deduce $T_1=T_{1,1}+T_{1,2}$, where since ${\rm div\,} u=0$ 
\begin{align*} 
 T_{1,1} = &-  \sum_{j,k=1}^3 \int_{D} \partial_k \Big( \sum_{i=1}^3 \partial_i u_i\Big)\,  u_j  \, \partial_k u_j dx =0, \\
T_{1,2}=&-  \sum_{i,j,k=1}^3 \int_{D} \partial_k u_i \,  u_j \,  \partial_i\partial_k u_j dx. 
\end{align*}
The inequality \eqref{fgh2} applied with $f=u_j$, $g=\partial_k u_i$ and $h=\partial_i \partial_k u_j$ implies
\[ |T_{1,2}|\leq \sum_{i,j,k=1}^3 \varepsilon_0 \| \partial_i\partial_k u_j\|_{L^2}^2 +  \sum_{i,j,k=1}^3 \frac{\varepsilon_1}{4\varepsilon_0}
\big\| |u_j|^\alpha \partial_k u_i \|_{L^2}^2 
+  \sum_{i,j,k=1}^3 \frac{C_\alpha}{\varepsilon_0 \varepsilon_1^{\frac{1}{\alpha-1}}} \| \partial_k u_i\|_{L^2}^2.\]
This completes  the proof of \eqref{1.18}.

(ii) Let $\alpha =1$ and $u\in X_1$. Then an integration by parts implies 
 \[  \langle  A^{1/2} B(u,u)\, , \,A^{1/2} u \rangle = \sum_{i,j,k=1}^3 \int_{D} \partial_k\big[ u_i \, \partial_i u_j \big]\,  \partial_k u_j dx 
 =-  \sum_{i,j=1}^3  \int_{D}  u_i \, \partial_i u_j \; \Delta u_j \, dx.
  \] 
The Cauchy-Schwarz  and Young inequalities imply \eqref{1.18Bis}. 
\end{proof}
\bigskip

For $\varphi \in X_0$,  set
\begin{equation}		\label{defF}
F(\varphi)= -\nu A\varphi - B(\varphi, \varphi) -a  \Pi |\varphi|^{2\alpha} \varphi.
\end{equation}

Lemma 2.2 page 415 in \cite{BaLi} provides upper and lower bounds of the non linear Brinkman Forchheimer term. 
 Let $\alpha \in [1,\infty)$;  there exist positive constants $C$ and $\kappa$ such that for $u,v\in \RR^3$
\begin{align}			\label{1.9}
\big| |u|^{2\alpha} u - |v|^{2\alpha} v \big| & \leq C |u-v|\, \big( |u|^{2\alpha} + |v|^{2\alpha} \big), \\
\big( |u|^{2\alpha} u - |v|^{2\alpha} v \big) \cdot (u-v) & \geq  \kappa |u-v|^2 \big( |u|+|v|\big)^{2\alpha}. 
 \label{1.10}
\end{align}

The following lemma gives upper bounds of $F$ for any $\alpha \in [1,\infty)$.
\begin{lemma}		\label{lem1.3}  Let $\alpha \in [1,+\infty)$. \\
 (i) Let $u\in X_0$, $v\in L^2(0,T;V) \cap L^{2\alpha+2}([0,T]\times D;\RR^3)$.  Then 
\begin{align}			\label{1.20}
\int_0^T\!\!  | \langle F(u(t)), v(t)\rangle | dt& \leq  C\big[ \|v\|_{L^2(0,T;V)} \|u\|_{L^2(0,T;V)} + \|v\|_{L^{2\alpha+2}([0,T]\times D;\RR^3)}
\|u\|_{L^{2\alpha+2}([0,T]\times D;\RR^3)}^{2\alpha+1} \nonumber  \\
& + \|v\|_{L^2(0,T;V)}\;  \underset{{t\in [0,T]}}{\mbox{\rm ess sup }}
 \|u(t)\|_{H}^{\frac{\alpha-1}{\alpha}} \|u\|_{L^{2\alpha+2}([0,T]\times D ; \RR^3)}^{\frac{\alpha +1}{\alpha}} T^{\frac{\alpha-1}{2\alpha}}\big]
\end{align}
for some positive constant $C$.\\
(ii) Let $u\in {\mathcal X}_0$, $v\in L^4(\Omega;L^2(0,T;V))\cap  L^{2\alpha+2}(\Omega_T\times D;\RR^3)$. Then  
\begin{align}			\label{1.21}
\EE\int_0^T \! |\langle &F(u(t)),v(t)\rangle | dt \leq C\Big[ \|v\|_{L^2(\Omega_T;V)} \|u\|_{L^2(\Omega_T;V)} +
 \|v\|_{L^{2\alpha+2}(\Omega_T \times D;\RR^3)}
\|u\|_{L^{2\alpha+2}(\Omega_T\times D;\RR^3)}^{2\alpha+1} \nonumber  \\
& + \|v\|_{L^4(\Omega; L^2(0,T;V))} \Big\{  \EE\Big(  \underset{{t\in [0,T]}}{\mbox{\rm ess sup }}
 \|u(t)\|_{H}^4\Big)\Big\}^{\frac{\alpha-1}{\alpha}} \|u\|_{L^{2\alpha+2}(\Omega_T\times D ; \RR^3)}^{\frac{\alpha +1}{\alpha}} 
 T^{\frac{\alpha-1}{2\alpha}}\Big]
\end{align} 
for some positive constant $C$. 
\end{lemma}
\begin{proof}
Integration by parts and the Cauchy-Schwarz inequality imply
\[ \nu \int_0^T \! \! | \langle  A u(t)\, , \, v(t)\rangle | dt = \int_0^T\! \Big| -
\nu \int_{D} \!  A^{\frac{1}{2}}  u(t,x)  A^{\frac{1}{2}}  v((t,x) dx \Big| dt
 \leq \nu \|u\|_{L^2(0,T;V)}
\|v\|_{L^2(0,T;V)}.
\]

Furthermore, H\"older's inequality with conjugate exponents $2\alpha+2$ and $\frac{2\alpha+2}{2\alpha+1}$ yields
\[ \int_0^T \Big| \int_{D} |u(t,x)|^{2\alpha} u(t,x) v(t,x) dx \Big| dt \leq 
\big\| |u|^{2\alpha} u\|_{L^{\frac{2\alpha +2}{2\alpha +1}}([0,T]\times D;\RR^3)}
\|v\|_{L^{2\alpha +2}([0,T]\times D;\RR^3)}.\]
Using the above upper estimates with  the inequality \eqref{1.11} concludes the proof of \eqref{1.20}. 

(ii) The upper estimate \eqref{1.21} is a straightforward consequence of  the upper estimates \eqref{1.13}, \eqref{1.20}, 
 the Cauchy-Schwarz and H\"older inequalities.
\end{proof}
\medskip

The next lemma provides estimates of the gradient of $F(u)$ for $\alpha \in [1,+\infty)$. Note that when $\alpha=1$, this requires that the coefficient
$a$ in front of the Brinkman-Forchheimer smoothing term is ``not too small" compared to the viscosity $\nu$.
\begin{lemma}			\label{lem1.4}
(i) Let  $\alpha >1$.  
For $\eta \in (0,\nu)$, $\tilde{a}\in (0,a)$, there exists a positive constant $C:=C(\alpha, \eta, \tilde{a})$ such that for  $u\in X_1$
 and $t\in [0,T]$,
\begin{align} 			\label{1.22}
\int_0^t \!\! \langle & A^{1/2}  F(u(s))  ,   A^{1/2}u(s) \rangle ds  \nonumber \\
&\leq -\eta \! \int_0^t\!\!  \| A u(s)\|_{\LL^2}^2 ds
 - \tilde{a}\!  \int_0^t \!\! \big\| |u(s)|^{\alpha} \nabla u(s) 
\big\|_{\LL^2}^2 ds + C\! \int_0^t \!\! \| \nabla u(s)\|_{\LL^2}^2 ds.
\end{align}
(ii) Let  $\alpha =1$ and suppose  $4\nu a >1$. Then for $\eta \in \big( 0, \nu - \frac{1}{4a}\big)$ and 
$\tilde{a} = a-\frac{1}{4(\nu-\eta)}$ we have
\begin{equation} 	\label{1.22Bis}
\int_0^t \!\! \langle  A^{1/2}  F(u(s))  ,   A^{1/2}u(s) \rangle ds \leq -\eta \! \int_0^t\!\!  \| A u(s)\|_{\LL^2}^2 ds
 - \tilde{a}\!  \int_0^t \!\! \big\| |u(s)|^{\alpha} \nabla u(s) \big\|_{\LL^2}^2 ds . 
\end{equation}
\end{lemma}
\begin{proof}
(i) Let $\alpha \in (1,\infty)$. For  $u\in X_1$, integration by parts implies for a.e. $s\in [0,t]$,
\[ \nu \langle  A^{\frac{1}{2}}  \Delta u(s),  A^{\frac{1}{2}}   u(s)\rangle = -\nu \| A  u(s)\|_{\LL^2}^2.\]
Furthermore,  
\begin{align}	\label{*}
\int_{D} \!\! \nabla \big( |u(s)|^{2\alpha} u(s)\big) \cdot \nabla u(s) dx = & \int_{D}\big[  |u(s)|^{2\alpha} \nabla u(s) \cdot \nabla u(s) 
+ 2\alpha |u(s)|^{2(\alpha -1)} \big( u(s)\cdot \nabla u(s) \big)^2 \big] dx	\nonumber \\
\geq & \int_{D}  |u(s)|^{2\alpha} \nabla u(s) \cdot \nabla u(s) dx  = \big\| |u(s)|^\alpha \nabla u(s) \big\|_{\LL^2}^2.
\end{align}
Hence, using   \eqref{1.18} with $C\, \varepsilon_0\in (0, \nu-\eta)$, then $\varepsilon_1$ 
such that $C\, \frac{\varepsilon_1}{4\varepsilon_0} \in
(0, a-\tilde{a})$, we deduce that for a.e. $s\in [0,T]$, 
\begin{equation}		\label{***}
 \langle  A^{1/2}  F(u(s))  ,   A^{1/2}u(s)\rangle \leq -\eta \| Au(s)\|_{\LL^2}^2 - \tilde{a} \big\| |u(s)|^\alpha \nabla u(t)\|_{\LL^2}^2 +
C(\alpha, \eta, \tilde{a}) \| \nabla u(s)\|_{\LL^2}^2.
\end{equation} 
Integrating this inequality on the time interval $[0,t]$ concludes the proof of \eqref{1.22}.

(ii) Let $\alpha=1$. Then using \eqref{1.18Bis} and \eqref{*}, we deduce for $\epsilon >0$  and $s\in [0,T]$
\[ \langle  A^{1/2}  F(u(s))  ,   A^{1/2}u(s)\rangle \leq -(\nu-\epsilon) \|A u(s)\|_{\LL^2}^2 + \frac{1}{4\epsilon} \| |u(s)| \nabla u(s)\|_{\LL^2}^2 -a 
\| |u(s)| \nabla u(s)\|_{\LL^2}^2.\]
Since $4a\nu >1$, for  $\eta \in \big( 0, \nu - \frac{1}{4a}\big)$ and $\epsilon = \nu-\eta$ and  $\tilde{a} = a-\frac{1}{4(\nu-\eta)}$ 
we deduce
\begin{equation}		\label{4*}
 \langle  A^{1/2}  F(u(s))  ,   A^{1/2}u(s) \rangle \leq -\eta \| Au(s)\|_{\LL^2}^2 - \tilde{a} \big\| |u(s)|^\alpha \nabla u(s)\|_{\LL^2}^2 .
\end{equation} 
 Integrating on the time interval $[0,t]$, we deduce \eqref{1.22Bis}.               
\end{proof} 
\bigskip

We finally prove upper estimates of increments $F(u)-F(v)$ for $\alpha \in [1,\infty)$.
\begin{lemma} 			\label{lem1.5}
There exists a positive constant $\kappa$ depending on $\alpha \in [1,+\infty)$, and for $\eta \in (0,\nu)$ a positive constant $\bar{C}(\eta)$, such that
for $u,v\in V\cap L^{2\alpha +2}(D;\RR^3)$,
\begin{equation} 				\label{1.26}
\langle F(u)-F(v), u-v\rangle \leq - \eta \| \nabla (u-v)\|_{\LL^2}^2 - a\kappa \big\| (|u|+|v|)^\alpha (u-v)\big\|_{\LL^2}^2 + 
\bar{C}(\eta) \| \nabla v\|_{\LL^2}^4 \| u-v\|_{\LL^2}^2.
\end{equation}
\end{lemma}
\begin{proof} 
Using integration by parts, we obtain $\nu \langle \Delta(u-v),u-v\rangle = -\nu \| \nabla(u-v)\|_{\LL^2}^2$. The monotonicity property 
\eqref{1.10} implies 
\begin{equation} 	\label{1.27}
 a\int_{D} \big( |u(x)|^{2\alpha} u(x) - |v(x)|^{2\alpha} v(x)\big) \cdot \big( u(x)-v(x)\big) dx \geq 
a \kappa \big\| (|u|+|v])^\alpha (u-v)\big\| _{\LL^2}^2.
\end{equation}
Finally, H\"older's inequality and the Gagliardo-Nirenberg inequality \eqref{GagNir} for the $\LL^4$ norm imply
\begin{align*}
|\langle B(u,u) - B(v,v) , u-v\rangle | =& |\langle B(u-v,v) , u-v\rangle |\\
\leq & \|u-v\|_{\LL^4}^2 \|\nabla v\|_{\LL^2} \leq \bar{C}_4^2 \| u-v\|_{\LL^2}^{\frac{1}{2}} \| \nabla (u-v)\|_{\LL^2}^{\frac{3}{2}} \| \nabla v\|_{\LL^2} \\
\leq & \frac{3}{4} \varepsilon^{\frac{4}{3}} \| \nabla (u-v)\|_{\LL^2}^2 + \frac{1}{4} \frac{1}{\varepsilon^4} \bar{C}_4^8 \|\nabla v\|_{\LL^2}^4 \| u-v\|_{\LL^2}^2,
\end{align*}
where the last inequality holds for any $\varepsilon >0$ by Young's inequality. Choosing $\frac{3}{4} \varepsilon^{\frac{4}{3}} \in (0, \nu-\eta)$,
we conclude the proof of \eqref{1.26}.
\end{proof} 

We next prove that  \eqref{3D-NS} has a unique strong solution   in $\mathcal{X}_1$. The outline is quite classical, based on some
 Galerkin approximation and a priori estimates. 
 
 \subsection{Galerkin approximation and a priori estimates} 		\label{Ap2}
   Recall that $D$ is periodic domain of $\RR^3$.   Let $(e_n, n\geq 1)$ be the orthonormal basis
of   $H$  defined in section \ref{def-WG} (that is made of functions in $H$ which are also orthogonal in $V$).  
For every integer $n\geq 1$ we set ${\mathcal K}_n:=\mbox{\rm span} (\zeta_1, \cdots , 
\zeta_n)$. Let $\Pi_n$ denote the projection from $K$  onto  $Q^{1/2}({\mathcal K}_n)$,  and 
let $W_n(t)=\sum_{j=1}^n \sqrt{q_j} \zeta_j \beta_j(t) = \Pi_n W(t)$.

Recall that if ${\mathcal H}_n= \mbox{\rm span} (e_1, ..., e_n)$, the orthogonal projection $P_n$  of  $H$ onto   ${\mathcal H}_n$ 
restricted to $V$ coincides with the
orthogonal projection  of  $V$ onto ${\mathcal H}_n$. 

Fix $n\geq 1$ and consider the following stochastic ordinary differential equation
on the $n$-dimensional space ${\mathcal H}_n$  defined  by  $u_{n}(0)=P_n u_0$,  and for
$t\in [0,T]$ and  $v \in {\mathcal H}_n$:
 \begin{equation} \label{un_H}
d(u_{n}(t), v)= \big\langle P_n F(u_{n}(t)),v\big\rangle  dt  
+  (P_n\,  G( u_{n}(t)) \, \Pi_n\,  dW(t), v), \quad \PP\, \mbox{\rm a.s.}, 
\end{equation}
where $F$ is defined in \eqref{defF}. 
 Then for $k=1, \, \cdots, \, n$ we have for $t\in [0,T]$:
 \[d(u_{n}(t), e_k)= \big\langle P_n F(u_{n}(t)),e_k \big\rangle\,  dt
+ \sum_{j=1}^n q_j^{\frac{1}{2}} \big( P_n\,  G(u_{n}(t))  \zeta_j\, ,\, e_k \big) \,  d\beta_j(t), \quad \PP \, \mbox{a.s.}
\] 
Note that  for $v\in {\mathcal H}_n$   the map
$  u\in  {\mathcal H}_n  \mapsto \langle F(u) \, ,\, v\rangle  $  
is locally Lipschitz. Indeed, $\HH^2\subset \LL^{2\alpha +2}$ and there exists some constant $C(n)$
 such that  $\|v\|_{\HH^2}  \leq C(n) \|v\|_{\LL^2}$ 
 for $v\in {\mathcal H}_n$. 
Let $\varphi, \psi, v\in {\mathcal H}_n$; 
integration by parts implies that 
 \[ 
|\langle \Delta \varphi - \Delta \psi ,v\rangle|  \leq 
 \|\varphi - \psi\|_V \, \|v\|_V   
 \leq  C(n)^2 \|\varphi - \psi\|_{\LL^2}\, \|v\|_{\LL^2}.
 \]  
 In the polynomial nonlinear term, the upper estimate \eqref{1.9}, the H\"older inequality with exponents $\frac{\alpha +1}{\alpha}$, $2\alpha +2$,  and
 $2\alpha +2$,  and the Sobolev embedding $\HH^2\subset \LL^{2\alpha+2}$ imply  
 \begin{align*}
  \Big| \int_{D}\!  \big(  |\varphi (x) |^{2\alpha}\varphi (x)- &|\psi (x)|^{2\alpha} \psi (x) \big) v(x) dx  \Big|   \\
 & \leq C\, 
\big( \|\varphi\|_{\LL^{2\alpha +2}}^{2\alpha }  + \|\psi\|_{\LL^{2\alpha +2}}^{2\alpha } \big)
\,  \|\varphi - \psi\|_{\LL^{2\alpha +2}}\,  \|v\|_{\LL^{2\alpha +2}} 
 \\
&\leq C \, C(n)^{2(\alpha +1)} \big( \|\varphi \|_{\LL^2}^{2\alpha} +\|\psi \|_{\LL^2}^{2\alpha}\big) \, \|\varphi - \psi\|_{\LL^2}\,  \|v\|_{\LL^2} .
\end{align*} 

Finally,  using  integration by parts, the  H\"older and Gagliardo-Nirenberg inequalities, we deduce:
\begin{align*}
 |\langle & B(\varphi ,\varphi) - B(\psi,\psi ), v\rangle|  
 =\big| -\langle B(\varphi -\psi , v)\, , \, \varphi \rangle - \langle B(\psi, v)\, , \, \varphi-\psi \rangle \big| \\
 & \leq 
C\, \|\varphi - \psi\|_{\LL^4} \big( \|\varphi \|_{\LL^4}+ \| \psi\|_{\LL^4} \big)  \|\nabla v\|_{\LL^2} 
\leq C C(n)^3 \|\varphi - \psi\|_{\LL^2} \big( \|\varphi\|_{\LL^2}+\|\psi\|_{\LL^2}\big) \|v\|_{\LL^2}.
\end{align*}  

 Condition ({\bf G})  implies that the
map    $u\in {\mathcal H}_n  \mapsto  \big( \sqrt{q_j}\, \big(G(u) \zeta_j\, ,\, e_k\big) : 1\leq
j,k\leq n \big)$ 
satisfies the classical global linear growth and Lipschitz conditions from ${\mathcal H}_n$ to $n\times n$ matrices uniformly in $t\in [0,T]$.
Hence by a well-known result about existence and
uniqueness of solutions to
stochastic  differential equations  (see e.g. \cite{Kunita}), there
exists a maximal solution  $u_{n}=\sum_{k=1}^n (u_{n}\, ,\, e_k\big)\, e_k\in {\mathcal H}_n$ to \eqref{un_H},
i.e., a stopping time
$\tau_{n}^*\leq T$ such that \eqref{un_H} holds for $t< \tau_{n}^*$ and if $\tau_n^* <T$, $\|u_{n}(t)\|_{L^2} \to \infty$ as $t \uparrow \tau_{n}^*$.
\par
The following proposition shows that $\tau_n^*=T$ a.s., and provides a priori estimates on norms of $u_n$, which do not depend on $n$. 
\begin{prop}			\label{apriori}
 Let $\alpha \in [1,\infty)$, and if $\alpha=1$, suppose that $4\nu a >1$. 

\noindent (i) Let $u_0$ be ${\mathcal F}_0$-measurable such that  $\EE\big( \|u_0\|_H^2\big)<\infty$,  
$T>0$ and $G$ satisfy \eqref{growthG_H} and
\eqref{LipG}. Then the evolution equation \eqref{un_H} with initial condition $P_n u_0$ has a unique global solution on $[0,T]$ (i.e., $\tau_n^*=T$
a.s.) with a modification  $u_n\in C([0,T];{\mathcal H}_n)$. Furthermore, if $\EE\big( \|u_0\|_H^{2p}\big) <\infty$ for some $p\in  [1,\infty)$, we have
$u_n \in {\mathcal X}_0$  and  
\begin{equation}			\label{apriori_H}
\sup_n \EE\Big( \sup_{t\in [0,T]} \|u_n(t)\|_H^{2p} + \int_0^T \big[ \|u_n(t)\|_V^2 + \|u_n(t)\|_{\LL^{2\alpha +2}}^{2\alpha +2}\big] 
\|u_n(t)\|_H^{2p-2}
\,  dt \Big) \leq C \big[ 1+\EE(\|u_0\|_H^{2p}) \big]. 
\end{equation}
(ii) If $\EE(\|u_0\|_V^{2p}) <\infty$ for some $p\in [1,\infty)$ and $G$ satisfies also \eqref{growthG_V}, we have furthermore
\begin{align}		\label{apriori_V}
\sup_n \EE\Big( \sup_{t\in [0,T]} \|u_n(t)\|_V^{2p} &+ \int_0^T \big[ \|A  u_n(t)\|_{\LL^2}^2 +  \||u_n(t)|^\alpha \nabla u_n(t)\|_{\LL^2}^2 \big]
\|u_n(t)\|_V^{2p-2} dt\Big) \nonumber \\
& \leq C \big[ 1+\EE(\|u_0\|_V^{2p}) \big] . 
\end{align}
\end{prop}
\begin{proof}

(i) For fixed $N>0$ set $\tau_N:= \inf\{ t\geq 0 : \|u_n(t)\|_H \geq N\} \wedge \tau_n^*$. It\^o's formula and the antisymmetry
property of  $B$ imply
\begin{equation}		\label{Ito_H^2}
\|u_n(t\wedge \tau_N)\|_H^2 = \| P_n u_0\|_H^2 -2\int_0^{t\wedge \tau_N} \!\! \big[ \nu \| \nabla u_n(s)\|_{\LL^2}^2 
+a \|u_n(s)\|_{L^{2\alpha +2}}^{2\alpha +2} \big] ds  + \sum_{i=1}^2 T_i(t),
\end{equation}
where
\begin{align*}
T_1(t)=& \; 2 \int_0^{t\wedge \tau_n} \big( G(u_n(s))\, dW_n(s)\, , \, u_n(s)\big), \quad
 T_2(t)= \; \int_0^{t\wedge \tau_n} \| P_n G(u_n(s)) \Pi_n \|_{\mathcal L}^2 \; ds. 
\end{align*}
Apply once more the It\^o formula to $z\mapsto z^p$ and $z=\|u_n(t\wedge \tau_N)\|_H^2$ for $p\in [2,\infty)$. We obtain
\begin{align}		\label{unp_H}
\|u_n(t\wedge \tau_N)\|_H^{2p} =&\;  \| P_n u_0\|_H^{2p} -2p\int_0^{t\wedge \tau_N} \!\! \big[ \nu  \| \nabla u_n(s)\|_{\LL^2}^2 
+a  \|u_n(s)\|_{L^{2\alpha +2}}^{2\alpha +2}\big] \|u_n(s)\|_H^{2p-2} ds \nonumber \\
& \;  + \sum_{i=1}^3 \bar{T}_i(t),
\end{align}
where
\begin{align*}
\bar{T}_1(t)= & \;  2p \int_0^{t\wedge \tau_N} \big( P_n G(u_n(s))\, dW_n(s) \, , \, u_n(s)\big) \, \|u_n(s)\|_H^{2p-2}, \\
\bar{T}_2(t)=&\; p\int_0^{t\wedge \tau_N} \| P_n G(u_n(s)) \Pi_n\|_{\mathcal L}^2 \,  \|u_n(s)\|_H^{2p-2}\, ds , \\
\bar{T}_3(t)=&\; 2p(p-1) \int_0^{t\wedge \tau_N} \|  \big( G(u_n(s))\, \Pi_n \big)^* u_n(s)\big\|_{K}^2\,  \|u_n(s)\|_H^{2p-4}\, ds.
\end{align*}
The growth condition \eqref{growthG_H} implies 
\[ \bar{T}_2(t) + \bar{T}_3(t) \leq p(2p-1) \int_0^{t} [K_0 + K_1 \|u_n(s\wedge \tau_N)\|_H^2] \, \|u_n(s\wedge \tau_N)\|_H^{2p-2}\, \mbox{\rm Tr } Q\, ds.\]
Using the Davis inequality, the growth condition \eqref{growthG_H} and Young's inequality, we deduce for  $\beta \in (0,1)$, 
\begin{align*}
\EE\Big(&  \sup_{s\leq t\wedge \tau_n}   \; \bar{T}_1(s)\Big) \leq \; 6p \; \EE\Big( \Big\{ \int_0^{t\wedge \tau_N} 
\|G(u_n(s))\|_{\mathcal L}^2 \; \|u_n(s)\|_H^{4p-2}\; 
 {\rm Tr}\,Q\, ds
\Big\}^{\frac{1}{2}} \Big) \\
\leq & \; \beta \; \EE\Big( \sup_{s\leq t}  \|u_n(s\wedge \tau_N)\|_H^{2p}\Big) 
+ \frac{9p^2}{\beta} \EE\!\! \int_0^t\!\! \big[ K_0+K_1 \|u_n(s\wedge \tau_N)\|_H^2\big] 
\|u_n(s\wedge \tau_N)\|_H^{2p-2} \, \mbox{\rm Tr } Q\, ds .
\end{align*}
Neglecting the first integral in the right hand side of \eqref{unp_H}, using the above upper estimates of $\bar{T}_i$ and the Gronwall lemma, we
deduce that for $\beta \in (0,1)$, 
\begin{equation}		\label{normp_unH}
\sup_{n\geq 1} \EE\Big( \sup_{s\leq T} \|u_n(s\wedge \tau_N)\|_H^{2p}\Big) \leq C(\beta,p, K_0, K_1, {\rm Tr}Q) \big[ 1+\EE(\|u_0\|_H^{2p}\big].
\end{equation}
As $N\to \infty$, the sequence of stopping times $\tau_N$ increases to $\tau^*_n$ and on the set $\{ \tau^*_n<T\}$, we have $\sup_{s\in [0, \tau_N]}
\|u_n(s)\|_H \to \infty$. Hence \eqref{normp_unH} implies $P(\tau_n^*<T)=0$ and for almost every $\omega$, for $N(\omega)$ large enough we
have $\tau_{N(\omega)}(\omega)=T$. Plugging the upper estimate \eqref{normp_unH} in \eqref{unp_H}, we conclude the proof of \eqref{apriori_H}.

Note that the above argument based on \eqref{Ito_H^2} instead of \eqref{unp_H} proves that if  $\EE(\|u_0\|_H^2) <\infty$ we have once more
$\tau_{N(\omega)}(\omega)=T$ for $N(\omega)$ large enough and a.e. $\omega$,  and that \eqref{apriori_H} holds for $p=1$.

We next prove that  $u_n\in {\mathcal X}_0$. 
Plugging the above upper estimate for $p=1$ in \eqref{Ito_H^2}, taking expected values and using Condition \eqref{growthG_H}, we obtain
\[ \EE\int_0^T \big[ \|u_n(s)\|_V^2 + \|u_n(s)\|_{L^{2\alpha +2}}^{2\alpha +2} \big] ds <\infty.\]
A similar argument using \eqref{normp_unH} in \eqref{unp_H} completes the proof of \eqref{apriori_H} when the $H$-norm of the initial condition 
has $2p$ moments.

(ii) Taking the gradient of both hand sides of \eqref{un_H}, using the It\^o formula and \eqref{grad_Pn}, we deduce for 
$    \tilde{\tau}_N:=\inf\{ s\geq 0\, : \, \|u_n(s)\|_V\geq N\} \wedge T$, 
\begin{align*}
\| A^{\frac{1}{2}} u_n(t\wedge &\tilde{\tau}_N)\|_{\LL^2}^2  = \|A^{\frac{1}{2}}  P_n u_0\|_{\LL^2}^2 
+ 2 \int_0^{t\wedge \tilde{\tau}_N} \langle A^{\frac{1}{2}}  P_n F(u_n(s)) ,
A^{\frac{1}{2}}  u_n(s)\rangle \, ds \\
& + 2\! \int_0^{t\wedge \tilde{\tau}_N} \!\! \big( A^{\frac{1}{2}}  P_n G(u_n(s)) dW_n(s) , A^{\frac{1}{2}} u_n(s)\big) + 
\int_0^{t\wedge \tilde{\tau}_N} \!\! \| A^{\frac{1}{2}}  P_n G(u_n(s)) \Pi_n\|_{\mathcal L}^2 \, ds\\
=& \, \|A^{\frac{1}{2}}  P_n u_0\|_{\LL^2}^2 + 2 \int_0^{t\wedge \tilde{\tau}_N} \langle A^{\frac{1}{2}}   F(u_n(s)) ,
A^{\frac{1}{2}}  u_n(s)\rangle \, ds \\
& + 2\! \int_0^{t\wedge \tilde{\tau}_N} \!\! \big( A^{\frac{1}{2}}   G(u_n(s)) \Pi_n dW(s) ,  A^{\frac{1}{2}}  u_n(s)\big) + 
\int_0^{t\wedge \tilde{\tau}_N} \!\! \| A^{\frac{1}{2}}  P_n G(u_n(s)) \Pi_n\|_{\mathcal L}^2 \, ds. 
\end{align*}
 Indeed, since $u_n(s)\in V$ for $s\leq t\wedge \tilde{\tau}_N$, we have $A^{\frac{1}{2}} u_n(s)\in H$ and $A^{\frac{1}{2}} G(u_n(s))\in
{\mathcal L}(K,H)$. 

Using once more the It\^o formula for the function $z\mapsto z^p$ for $p\in [2,\infty)$, we obtain
\begin{align}		\label{Ito_unV_p}
\| A^{\frac{1}{2}}  u_n(t\wedge \tilde{\tau}_N)\|_{\LL^2}^{2p}   \leq  &\; \|A^{\frac{1}{2}}   u_0\|_{\LL^2}^{2p}  + 2 \int_0^{t\wedge \tilde{\tau}_N} 
\langle A^{\frac{1}{2}}  \nabla F(u_n(s)) ,  A^{\frac{1}{2}}  u_n(s)\rangle \, \| A^{\frac{1}{2}} u_n(s)\|_{\LL^2}^{2(p-1)}  ds  \nonumber \\
&\; + \sum_{i=1}^3 \tilde{T}_i(t) ,
\end{align} 
where
\begin{align*}
\tilde{T}_1(t)&=\; 2p \int_0^{t\wedge \tilde{\tau}_N} \big( A^{\frac{1}{2}}  G(u_n(s)) dW_n(s)\, , \, A^{\frac{1}{2}}  u_n(s)\big) \, 
 \| A^{\frac{1}{2}}  u_n(s)\|_{\LL^2}^{2(p-1)} ds,\\
\tilde{T}_2(t)&=\; p\int_0^{t\wedge \tilde{\tau}_N} \| G(u_n(s) ) \Pi_n\|_{\tilde{\mathcal L}}^2 \; \| A^{\frac{1}{2}}  u_n(s)\|_{\LL^2}^{2(p-1)} ds, \\
\tilde{T}_3(t)&=\; 2p(p-1)\int_0^{t\wedge \tilde{\tau}_N} \| \big(   A^{\frac{1}{2}}  G(u_n(s))  \, \Pi_n  \big)^* A^{\frac{1}{2}}  u_n(s)\|_{K}^2 \; 
 \| A^{\frac{1}{2}}  u_n(s)\|_{\LL^2}^{2(p-2)}\, ds.
\end{align*}
 Since 
\[ 
\| \big(A^{\frac{1}{2}}  G(u_n(s))  \, \Pi_n  \big)^*\|_{{\mathcal L}(H;K)} \leq \|A^{\frac{1}{2}}  G(u_n(s))\|_{{\mathcal L}(K;H)} \leq 
\| G(u_n(s))\|_{{\mathcal L}(K;V)},\] 
 the growth condition \eqref{growthG_V} and Young's inequality imply
\[ \tilde{T}_2+\tilde{T}_3(t) \leq C(p,T,\mbox{\rm Tr } Q, \tilde{K}_0, \tilde{K}_1) 
\Big[ 1+\int_0^{t\wedge \tilde{\tau}_N} \|u_n(s)\|_H^{2p} + \| \nabla u_n(s)\|_{\LL^2}^{2p} ds\Big].
\] 
The growth condition \eqref{growthG_V}, the Gundy  and Young inequalities imply that for $\tilde{\beta}\in (0,1)$, 
\begin{align*}
\EE\Big( \sup_{s\leq t} \tilde{T}_1(s) \Big) \leq &\;  C(p) \EE\Big( \Big\{ \int_0^{t\wedge \tilde{\tau}_N} \big[ \tilde{K}_0 + \tilde{K}_1 \| u_n(s)\|_V^2\big]
\; \| \nabla u_n(s)\|_{\LL^2}^{4p-2} \, \mbox{\rm Tr } Q\, ds \Big\}^{\frac{1}{2}} \Big) \\
\leq &\;  \tilde{\beta} \EE\Big( \sup_{s\leq t} \|u_n(s\wedge \tilde{\tau}_N)\|_{\LL^2}^{2p} \Big) 
+ \tilde{\beta} \EE\Big( \sup_{s\leq t} \|\nabla u_n(s\wedge \tilde{\tau}_N)\|_{\LL^2}^{2p}  \Big) \\
&\; + C(\tilde{\beta}, \mbox{\rm Tr } Q, \tilde{K}_0, \tilde{K}_1) \Big[ 1+ \EE
\Big( \int_0^t \| \nabla u_n(s\wedge \tilde{\tau}_N)\|_{\LL^2}^{2p} ds\Big) \Big].
\end{align*} 

\noindent Let $\rho\in (0,\nu)$ and $\tilde{a}\in (0,a)$. Using  \eqref{1.22} for $\alpha >1$ and \eqref{1.22Bis} for $\alpha =1$,
  \eqref{apriori_H} and the Gronwall lemma,  we deduce
 \[ \EE\Big( \sup_{s\leq \tilde{\tau}_N} \|u_n(s)\|_V^{2p}\Big)   \leq C\big[ 1+ \EE(\|u_0\|_V^{2p})\big].
 \] 
 for some positive constant $C$ which does not depend on $N$ and $n$. For fixed $n$, letting $N\to \infty$ and using the monotone
  convergence theorem we deduce 
 $u_n\in L^{2p}(\Omega; L^\infty(0,T;V))$. 
Plugging this  in  
\eqref{Ito_unV_p} and taking expected values, we conclude the proof of \eqref{apriori_V}. 
\end{proof}

\subsection{Proof of global well-posedness of the solution} \label{Ap3}

\noindent  The proof of Theorem \ref{th_gwp} is classical  and uses the upper estimates \eqref{1.12} and \eqref{1.14}
 for the uniqueness; see e.g.  
\cite{BeMi_anisotropic} for details.  
 \bigskip
 
 \noindent {\bf Acknowledgements.}   
 Annie~Millet's research has been conducted within the FP2M federation (CNRS FR 2036). 
\bigskip

\noindent{\bf Declarations.}

\noindent{\bf Funding.} Hakima Bessaih was partially supported by Simons Foundation grant: 582264 and NSF grant DMS: 2147189. \\
This   research was  started
while both authors stayed at the Mathematisches Forschung Institute Oberwolfach during a Research in Pairs program. 
They want to thank the MFO for the financial support and excellent working conditions.

\noindent {\bf  Conflicts of interest.} The authors have no conflicts of interest to declare that are
 relevant to the content of this article.

\noindent{\bf  Availability of data and material.} Data sharing not applicable to this article as
no datasets were generated or analysed during the current study.


\begin{thebibliography}{99}
  \bibitem{Adams} Adams, R.A., Sobolev spaces, New York, Academic Press 1975.  
  \bibitem{BaLi} Barret, J.W. \& Liu, W.B., Finite elements approximations for the parabolic $p$-Laplacian, {\it SIAM J. Numer. Anal.}, {\bf 31} 
  (1994), 413--428. 

\bibitem{Ben1}
Bensoussan A., Some existence results for stochastic partial differential equations, Pitman Res. Notes
Math. Ser., 268, Longman Sci. Tech., Harlow, (Trento, 1990),  37--53.

\bibitem {Ben2}
Bensoussan, A., Glowinski  R. \& Rascanu, A., Approximation of Some Stochastic Differential Equations
by Splitting Up Method,  {\it Applied Mathematics and Optimization},  {\bf 25} (1992), 81--106.
 
 \bibitem{BeBrMi} Bessaih, H., Brze\'zniak, Z. \& Millet, A., Splitting up method for the 2D stochastic Navier-Stokes equations,
 {\it Stochastic PDE: Analysis and  Computations} {\bf 2-4}  (2014), 433-470.
 
 \bibitem{BeMi_anisotropic}  Bessaih, H. \& Millet, A., On stochastic modified 3D Navier-Stokes equations with anisotropic viscosity, 
 {\it Journal of Mathematical
 Analysis and Applications} {\bf 462} (2018), 915--956. 

\bibitem{BeMi_multi}  Bessaih,  H. \& Millet, A., Strong $L^2$ convergence of time numerical schemes for the stochastic two-dimensional Navier-Stokes 
equations. {\it IMA J. Numer. Anal.}, {\bf  39-4}  (2019) 2135–2167. 

\bibitem{BeMi_FEM} Bessaih, \& H. Millet, A.,  Space-time  Euler discretization schemes  for the stochastic 2D Navier-Stokes equations,
Stochastic PDE: Analysis and  Computations, Published online 07 October 2021, https://link.springer.com/article/10.1007/s40072-021-00217-7. 


 \bibitem{BeMi_additive}  Bessaih, H. \& Millet, A., Strong rates of convergence  of space-time discretization schemes   
 for the  2D Navier-Stokes equations with additive noise,  Stochastics and Dynamics, Published online 26 January 2022, 
 https://doi.org/10.1142/S0219493722400056.  


  
 \bibitem{BTZ} Bessaih, H., Trabelsi, S.  \& Zorgati, H. , Existence and uniqueness of global solutions for the modified anisotropic 3D 
  Navier-Stokes equatiions. M2AN, 50 (2016), 1817–1823.
 
\bibitem{Breckner}
Breckner, H.,  Galerkin approximation and the strong solution of the Navier-Stokes equation, 
{\it J. Appl. Math. Stochastic Anal.}, {\bf 13-3}  (2000), 239--259.

 \bibitem{BrCaPr} Brze\'zniak, Z., Carelli, E. \& Prohl, A., Finite element base discretizations of the incompressible Navier-Stokes equations
 with multiplicative random forcing, {\it IMA J. Numer. Anal.}, {\bf 33-3}, (2013), 771--824.

\bibitem{CarPro}  Carelli, E. \& Prohl, A., Rates of convergence for discretizations of the stochastic incompressible Navier-Stokes equations,
{\it SIAM J. Numer. Anal.} {\bf 50-5}, (2012), 2467-2496.

\bibitem{CDGG} Chemin, J.-Y., Desjardin,   B.,Gallagher,  I.  \& Grenier,  E., Mathematical Geophysics: An Introduction to Rotating Fluids and the 
Navier-Stokes Equations, Oxford Lecture Series in Mathematics and its Applications 32 (2006).

\bibitem{ChuMil} Chueshov, I. \& Millet, A.,  {Stochastic 2D hydrodynamical type systems: Well posedness
and large deviations}, \emph{Appl. Math.  Optim.}, \textbf{61-3} (2010), 379--420.

\bibitem{DaPZab} Da Prato, G. \& Zabczyk, J., Stochastic Equations in infinite Dimensions, Cambridge University Press, 1992. 


\bibitem{Dor} D\"orsek, P., Semigroup splitting and cubature approximations for the stochastic Navier-Stokes Equations,  {\it SIAM J. Numer. Anal.}
{\bf 50-2} (2012), 729-746.

\bibitem{F}  Flandoli, F.,  A stochastic view over the open problem of well-posedness for the 3D Navier-Stokes equations. 
Stochastic analysis: a series of lectures, Progr. Probab., 68, Birkh\" auser/Springer, Basel, 2015, 221–246. 

\bibitem {FG}
Flandoli, F. \&  Gatarek, D., Martingale and stationary solutions for stochastic Navier-Stokes equations, 
{\it Probability Theory and Related Fields},  {\bf 102} (1995), 367--391. 

\bibitem{FM} Flandoli, F. \&  Mahalov, A. Stochastic three-dimensional rotating Navier-Stokes equations: averaging, convergence and regularity. 
{\it Arch. Ration. Mech. Anal.}, {\bf  205-1}  (2012), 195–237.

\bibitem{GigMiy} Giga, Y. \& Miyakawa, T., Solutions in $L_r$ of the Navier-Stokes Initial Value Problem, {\it  Archive for Rational
Mechanics and Analysis} {\bf 89-3} (1985) 267--281. 

\bibitem{GirRav} Girault, V. \& Raviart, P.A.   Finite Element Method for Navier-Stokes Equations: theory and algorithms, 
Springer-Verlag, Berlin, Heidelberg, New York (1981).


\bibitem{HutJen}   Hutzenthaler, M. \&  Jentzen, A., Numerical approximations of stochastic differential equations with 
non-globally Lipschitz continuous coefficients,  Mem. Amer. Math. Soc. 236 (2015), no. 1112.


\bibitem{Kunita} Kunita, H., Stochastic Flows and Stochastic Differential  Equations, Cambridge University Press, Cambridge-New York, 1990. 


\bibitem{Mohan-1} Mohan, M., Stochastic convective Brinkman-Forchheimer equations, arXiv:2007.09376 (2020). 



\bibitem{Pri} Printems, J., On the discretization in time of parabolic stochastic partial differential equations, {\it M2AN Math. Model. Numer. Anal. } 
{\bf 35-6}, (2001) 1055-1078.


\bibitem{Tem} Temam, R.,  Navier-Stokes equations. Theory and numerical analysis. 
Studies in Mathematics and its Applications 2, North-Holland Publishing Co.,
Amsterdam - New York (1979). 

\bibitem{Temam-84} Temam, R.,  Navier-Stokes equations and Nonlinear Functional Analysis, CBMS-NSF Regional Conference Series in Applied Mathematics, 66. Society for Industrial and Applied Mathematics (SIAM), Philadelphia, PA, (1995).

\end{thebibliography}
\end{document}